\def\ZZ         {{\mathbb Z}}
\def\RR         {{\mathbb R}}
\def\CC         {{\mathbb C}}

\def\HH         {{\mathbb H}}
\def\QQ         {{\mathbb Q}}
\def\PP         {{\mathbb P}}

\def\ZZ         {{\mathbb Z}}

\def\E           {{\cal E}}
\def\F          {{\cal F}}

\def\H         {{\cal H}}

\def\L          {{\cal L}}

\def\O          {{\cal O}}
 
\def\Q      {{\cal Q}}

\def\ii         {{\rm i}}
\def\ee         {{\rm e}}

\def\sin        {{\rm sin}}

\def\sinh       {{\rm sinh}}

\def\Ell        {{\cal ELL}}

\def\Prod       {{\displaystyle\prod}}

\def\cal        {\mathcal}

\documentclass{amsart}
\usepackage{amssymb}
\usepackage{verbatim}
\usepackage{eucal}
\usepackage{eufrak}
\newtheorem{theorem}{Theorem}[section]

\newtheorem{prop}[theorem]{Proposition}
\newtheorem{corollary}[theorem]{Corollary}

\theoremstyle{definition}
\newtheorem{dfn}[theorem]{Definition}

\newtheorem{example}[theorem]{Example}

\theoremstyle{remark}
\newtheorem{remark}[theorem]{Remark}

\title{Elliptic genus of singular algebraic varieties and quotients}

\author{Anatoly Libgober}
\address{Department of Mathematics\\
University of Illinois\\
Chicago, IL 60607 and 
National Science Foundation,
4201 Wilson Blvd. Arlington VA, 22230}
\email{libgober@math.uic.edu}

\thanks{Author supported by a grant from Simons Foundation}
\begin{document}
\begin{abstract} 
We discuss the basic properties of various versions of 
two variable elliptic genus with special attention to the equivariant 
elliptic genus. The main applications are to the elliptic genera
attached to non-compact GITs, including the elliptic genera of
Witten's phases on $N=2$ 
theories.
\end{abstract}

\maketitle

\section{Preface}

This paper provides an overview of the main results on complex elliptic genus while its  
 second part  focuses on equivariant elliptic genus and contains additional 
details regarding treatment of elliptic genera of phases of $N=2$ theories
given in \cite{phasesme}. The first two sections  give chronological
reviews of highlights of development of elliptic genus as well as relations 
to other problems since its introduction in 1980s (section
\ref{intro}) and recalls the key definitions along with the properties 
related to (complex, two variable) elliptic genera
(cf. section \ref{review}). Then we describe the 
equivariant elliptic genera using approach to  
equivariant cohomology given in \cite{edidin}. It gives a fast way 
to derive basic properties of equivariant elliptic genus obtained in
\cite{waelder} from non-equivariant version given in \cite{BLannals}.
The final sections review the properties of elliptic genera of 
Witten phases of $N=2$ theories (cf. \cite{n2}), following \cite{phasesme}, but also makes explicite 
specializations of elliptic genus to $\chi_y$-genus and the euler
characteristics, providing in Landau-Ginzburg instance new links between elliptic genus 
and invariant of singularities. Example \ref{lgexample} gives direct 
calculation of $\chi_y$-genus of LG phase (and can be a starting point
for reader interested on singularity theory) while the last section obtains
it as a specialization of elliptic genus. Two appendices record
well known information on basics of theta-functions and quasi-jacobi
forms introduced in \cite{mequasijacobi}.

\section{Introduction}\label{intro} Elliptic genus appeared in the middle of 80s 
in the works of topologists and physicists. 
In mathematics, it was viewed as either index of an operator
on a graded infinite dimensional vector 
bundle with finite dimensional graded components (cf.(\ref{formula1})
below) or (via Riemann-Roch) as a combination of characteristic
numbers (cf.(\ref{formula2})).
Motivation was the problem of finding 
rigid genera of differentiable
Spin-manifolds 
endowed with a circle action
(cf. \cite{landweberpaper},\cite{krichever}) extending 
Atiyah-Hirzebruch rigidity of $\hat A$-genus.   
In physics, elliptic genera appeared as
 indices of certain Dirac-like operators in free loop space 
associated with Spin-manifolds and also in connection with anomaly cancellations 
(cf. volume \cite{landweber}, \cite{wittenellgen} for overview of the first results 
 and references therein e.g. \cite{schekel}). 

Initial versions of elliptic genus were given in the context of differentiable
manifolds but complex versions of elliptic genus were proposed by F.Hirzebruch
(cf. \cite{hirzarticle}, see also \cite{krichever}, \cite{hohn}) 
and by E.Witten (\cite{landweber}) soon after.  
A different type of elliptic genus, associated with
$C^{\infty}$-manifolds with vanishing first Pontryagin class was
proposed by Witten (cf. \cite{landweber}). 
It lead to important connections with homotopy theory
and elliptic cohomology (cf. \cite{hopkins}, \cite{ando}).

This first period culminated with the proof of Witten's rigidity
conjecture by R.Bott and C.Taubes (cf.\cite{botttaubes}). 
In complex case rigidity has been proven by 
Hirzebruch and in \cite{krichever}
(for somewhat different but closely related versions of {\it complex}
elliptic genus).
Further study of rigidity was done in \cite{liurigid}.  
Much of material  from this first period   
is summarized in Hirzebruch's book \cite{hirzmod} to which we refer 
a reader. 

Elliptic genus, as an invariant of superconformal 
field theory (SCFT) (or of a representation of a superconformal algebra),  
was considered about the same time (cf. \cite{wittenlg} and references
to earlier works there). 
In the case of sigma models associated with a manifolds, an invariant of
SCFT becomes an invariant of underlying manifold. 
There are, however, other backgrounds with which one can associate SCFT
and obtain the invariants of such a background.
Such notable examples are the SCFTs which are  minimal models and
Landau-Ginzburg models. The latter are associated with weighted 
homogeneous polynomials with isolated singularities.   
Elliptic genus, as a character of representation of superconformal
algebra  in complex setting (i.e. $N=2$ SCFT) was
given in \cite{kawaiyang} following \cite{wittenlg}.

Mathematical version of the elliptic genus as an invariant of SCFT 
was presented in \cite{MSV} where the authors constructed the chiral deRham complex 
of a manifold and the associated vertex operator algebra (VOA). 
The characters of vertex operator algebras, which are close relatives 
of elliptic genus of SCFTs, were considered from the very beginning 
of the  study of VOAs (\cite{borcherds} \cite{mason}) though 
mathematical study of analogs of $N=2$ superconformal algebras
 still is not well developed at
the moment. 

In the late 90's, the focus of  mathematical study 
shifted to elliptic genera of singular varieties
(\cite{totaro}).
On one hand, it  was motivated by Goresky-MacPherson problem 
(cf.\cite{totaro},\cite{GM}) 
of determining the Chern numbers (rather than Chern classes, which 
being homological, are not a part of multiplicative structure and do not 
determine the Chern numbers) 
of singular varieties admitting a small resolution of singularities, 
which are independent of a resolution. 
The intersection homology groups, introduced in early 80s, 
do possess this property, cf. \cite{GM})) 

On the other hand, mirror symmetry is inherently interwoven 
with singular varieties: in its very first example 
 of smooth quintic in $\PP^4$ the mirror partner is 
{\it the orbifold} which is the global quotient of such quintic by the
action of abelian group of order 125 and exponent 5.
The relation between SCFTs, which is a physics definition 
of  mirror symmetry, implies the following relation 
between the elliptic genera of mirror partners (cf. \cite{wittenmirror})
\begin{equation}\label{mirror}
Ell(X)=(-1)^{dim X}Ell(X')
\end{equation} 

Mathematically, the relations similar to (\ref{mirror}) but involving
Hodge numbers, Gromov-Witten invariants, derived categories etc.
serve as either a definition or as a test of mirror symmetry.

The orbifold elliptic genus was proposed in the context of Landau-Ginzburg models
by Witten (cf. \cite{wittenlg}). In the context of manifolds (i.e. sigma-models) 
 the elliptic genus of orbifolds 
 apparently was understood in physics terms already right after introduction of 
orbifold euler characteristic (cf. \cite{dixon}, \cite{hirzebruchorbifold}).
Mathematical definition of the orbifold elliptic genus 
was given in \cite{Bduke}. This paper also contains 
a definition of elliptic genus for certain class of singular varieties, 
including the orbifolds,  
in terms of resolutions of singularities. The relation between
 both notions of elliptic genus of orbifolds, which is the so called McKay correspondence 
for elliptic genus, had been proven in \cite{BLannals}. It 
has as a very special case the numerical relations which in dimension 2,
are consequences of the relation between representations 
of finite subgroups of $SL_2(\CC)$ and resolutions of quotients
of $\CC^2$ due to J.McKay (cf. \cite{McKay}).

The identity  (\ref{mirror}) for the hypersurfaces in toric varieties 
corresponding to the dual polyhedra (Batyrev's mirror symmetry, cf.\cite{batyrev})
was shown in \cite{BLinvent}.
One of the major application of orbifold elliptic genus to the 
elliptic genera of Hilbert schemes of K3-surfaces 
was given in \cite{DMVV}. A vast generalization 
in the context orbfold elliptic genera of symmetric products 
was given in \cite{Bduke}. 

It is interesting to compare elliptic genus with the other invariants 
of smooth and singular varieties appearing in the context
of mirror symmetry: Hodge numbers, Gromov-Witten invariants,
and derived or Fukaya categories (cf .\cite{Konts}). The similar issues, as those  
mentioned above in the context of elliptic genus, e.g. search for 
extension of original definitions from smooth to singular varieties, 
behavior in mirror correspondence, MacKay correspondence etc.  
appeared in the study of all these invariants. 
However, results for one type rarely imply the results 
for others. For example, 
it is convenient to organize the Hodge numbers of smooth 
projective varieties into $E$-function: $E(u,v)=\sum h^{p,q}u^pv^q$.
It has the Hirzebruch's genus (cf. \cite{methods})
$\chi_y=\sum \chi(\Omega^p)y^p=\sum_{p,q} (-1)^qh^{p,q}y^p$
as specialization $y=u,v=-1$.
The elliptic genus is related to $\chi_y$-genus via:
$y^{-{{\rm dim} X } \over 2} \chi_{-y}(X)=lim_{q \rightarrow
  0}Ell(X)$.
However, neither $E$-function or $Ell(X)$ determine each other (cf. \cite{BLinvent}).
Both invariants factors through different universal rings of classes of
manifolds: the $K$-group of varieties, in the case of $E$-function,
and the group of unitary cobordisms, in the case of elliptic genus.
In fact, $E(u,v)$ is a homomorphism $K_{\CC}(Var) \rightarrow
\ZZ[u,v]$ while the elliptic genus is a homomorphism from 
the cobordism ring $\Omega^U \rightarrow \CC[a,b,c,d]$ to a certain 
polynomial algebra of functions on the product of $H \times \CC$ of upper and the
whole complex plane respectively.
Neither, appears to have an extension with good properties 
to a bigger ring (cf. however \cite{kachru}).  

An attempt to extend mathematical treatment of elliptic genus to a
wider context, which includes the elliptic genera of singular
varieties and Landau-Ginzburg models was made in \cite{phasesme}.
More specifically, for certain geometric invariants theory (GIT)
quotients one can define elliptic genus such that for 
the quotients considered by Witten in \cite{n2} and corresponding 
to Calabi Yau or Landau Ginzburg models they reproduce
respectively elliptic genera of Calabi Yau manifolds considered 
in mathematics and physics literature and the elliptic genera 
of Landau Ginzburg models considered in physics. 
The approach of \cite{phasesme} is based on use of equivariant elliptic genus
(in mathematics literature equivariant elliptic genus of compact varieties 
was considered by R.Waelder \cite{waelder}).  In particular it implies
LG/CY correspondence for elliptic genus as a consequence of 
equivriant McKay correspondence. In the following sections we 
spell out some of the details about elliptic genus of such GIT quotients.

\section{Review of previous work}\label{review}

\subsection{Complex manifolds}

The two variable elliptic genus, which is the subject of  
this paper, can be defined as the holomorphic euler characteristic 
of a bi-graded bundle associated with the manifold. More precisely,
given a vector
bundle $F$ on a complex manifold $X$, one associates with it   
the Poincare series $\Lambda_t F=\sum \Lambda^i(F)t^i, 
S_t(F)=\sum Sym^i(F)t^i$  in the ring of polynomials in formal
variable $t$ with coefficients
in the  semi-ring generated by vector bundles. With these notations,    
the elliptic genus is given by Fourier expansion with coefficients of 
monomials $q^iy^j$ being the holomorphic euler characteristics 
of bi-graded components of the infinite
tensor product of graded bundles with $q^iy^j$ providing the
bi-grading
 (cf.  \cite{krichever}, \cite{hohn},
\cite{totaro}, \cite{BLinvent}):
\footnote{H\"ohn \cite{hohn} uses $y \rightarrow -y$.}

\begin{equation}\label{formula1} Ell(X)=
y^{-{{\rm dim} X} \over 2}
\chi(X,\otimes_{n \ge 1} 
\Bigl(\Lambda_{-yq^{n-1}}\Omega^1_X \otimes \Lambda_{-y^{-1}q^n} 
T_{X} \otimes S_{q^n}\Omega^1_X \otimes
S_{q^n} T_X \Bigr) \otimes K_X^{-k}).
\end{equation}
(here $T_X, \Omega^1_X,K_X$ are respectively the tangent, cotangent
and canonical bundles of $X$ and $k$ is a constant).

Riemann-Roch theorem implies that (\ref{formula1}) is a linear combination 
of Chern numbers defined as follows. Evaluation of 
(\ref{formula1}) for  a compact complex manifold provides a
homomorphism of cobordism {\it ring} $\Omega^U$ of almost complex manifolds 
(cf. \cite{stong}). The target of this homomorphism is a ring of holomorphic 
functions on $\HH \times \CC$ where $\HH$ is the upper half-plane if one 
interprets the formal variable in (\ref{formula1}) as $q=e^{2\pi
  \sqrt{-1}\tau},
y=e^{2 \pi \sqrt{-1} z}, \tau \in H, z \in \CC$.  Hirzebruch's
formalism (cf. \cite{methods}) implies that any such a homomorphism $\phi: \Omega^U \rightarrow
R$ with values in a commutative ring $R$ (i.e. a $R$-valued genus)
can be specified by a formal power series $Q(x) \in R[[x]]$ so 
that $\phi(X)=\Prod Q(x_i)[X]$ is evaluation of the product series at the 
Chern roots $x_i$ on the fundamental class $[X] \in H_{{\rm dim_{\RR}}
  X}(X)$ of $X$ (the Chern roots $x_i$ satisfy $\Prod(1+x_i)=c(X)$ where
$c(X) \in  H^*(X)$ 
is the total Chern class of $X$).
In the case of (\ref{formula1}), $Q(x)$ is the Taylor series in
variables $x$ of the function 
${{x^{1-k}\theta( {x \over {2 \pi \sqrt{-1}}}-z,\tau)} \over {\theta({x \over {2
        \pi \sqrt{-1}}},\tau)}}$ and for  $k=0$ one has
\begin{equation}\label{formula2}
   Ell(X)=\Prod_i {{x\theta({x \over {2
        \pi \sqrt{-1}}}-z,\tau)} \over {\theta({x \over {2
        \pi \sqrt{-1}}},\tau)}}[X]
\end{equation}

The holomorphic functions, which are elliptic genera of manifolds
have important modularity properties.
If $c_1(X)=0$ then $Ell(X)$ is Jacobi form for semidirect
product of $SL_2(\ZZ)$ and $\ZZ^2$ (the Jacobi group) i.e. 
obeys the following transformation laws:
\begin{equation}
\phi({{a\tau+b} \over {c\tau+d}},{z \over {c\tau+d}})=
(c\tau+d)^k
e^{{2\pi \sqrt{-1} t c z^2} \over {c\tau+d}}\phi(z,\tau) 
\end{equation} 
$$\phi(\tau,z + \lambda \tau + \mu) =
(-1)^{2t(\lambda+\mu)}e^{−2\pi \sqrt{-1}t(\lambda^2 \tau+2\lambda z)}
\phi(\tau,z)$$
$$ \ \ \ a,b,c,d, \lambda, \mu \in \ZZ, ad-bc=1 
$$
Here $k,t \in \ZZ$ are weight and index respectively of the (weak)
Jacobi form $\phi(z,\tau)$.
For Calabi Yau manifold of dimension $d$, $Ell(X)$ given by
(\ref{formula1}) or (\ref{formula2}) is Jacobi form of weight zero and
index $d \over 2$. \footnote{different normalizations in (\ref{formula2}), 
used in some papers, may
  lead to a different weight and index.}

Without Calabi Yau condition (\ref{formula2}) is a
{\it quasi-Jacobi form} in the sense of \cite{mequasijacobi}
(cf. also Appendix II below) 
\footnote{D.Zagier pointed out that, at least for some of these
  functions, the term quasi-elliptic would be more
  appropriate}.
It follows (cf. \cite{mequasijacobi} theorem 2.12) that 
elliptic genera of almost complex manifolds 
are polynomials 
in $\hat E_2(z,\tau)=(E_2(z,\tau)-e_2(\tau)) ({{\theta(z,\tau)} \over {\theta'(0,\tau)}})^2
, \hat E_n(z,\tau)=E_n(z,\tau) ({{\theta(z,\tau)} \over {\theta'(0,\tau)}})^n
, n \ge 1$ where $E_n(z,\tau)$ are 
the two variable Eisenstein series (cf. Example \ref{eisenstein})
\begin{equation}
      E_n(z,\tau)=\sum_{a,b \in \ZZ^2}({1 \over {z+a\tau+b}})^n   \ \ \ n
      \in \ZZ, n \ge 1 
\end{equation} 
(with appropriate choice of summation order for $n=1,2$ cf. \cite{mequasijacobi})
and $e_2(\tau)$ is the one variable Eisenstein series.
\footnote{i.e. $E_2(0,\tau)$ with omitted summand corresponding to
  $a=b=0$.}

For example, the elliptic genus of a complex surface of degree $d$ in
$\PP^3$ can be calculated as
\begin{equation}
(E_1^2({1 \over 2}d^2-4d+8)d+
(E_2-e_2) ({d^2 \over 2}-2)d) ({{\theta(z,\tau)} \over {\theta'(0,\tau)}})^2
 \end{equation}
In particular for K3-surface, i.e. the case $d=4$, one obtains
$$24(E_2-e_2) ({{\theta(z,\tau)} \over {\theta'(0,\tau)}})^2$$

The elliptic genus considered in \cite{BLinvent} is given by
(\ref{formula2}) or 
(\ref{formula1}) with $k=0$. 
Up to a factor depending only on dimension (cf. \cite{BLinvent},
Prop.2.3), it coincides with the 
elliptic genus considered in \cite{krichever}, \cite{totaro}, \cite{hohn}.
The latter two works use Weierstrass
$\sigma$-function (cf.(\ref{sigmafunction}), in Appendix I) 
writing the characteristic series 
as follows\footnote{these papers use instead of $\Upsilon$ the 
notation $\Phi(x,\tau)$ which 
a little different than the one used in \cite{hirzmod}; notations here
and below
are the same as in \cite{hirzmod}.}. Let 
\begin{equation}
    \Upsilon(x,\tau)=e^{-G_2(\tau)x^2-{x \over 2}}\sigma(x,\tau)=
e^{-{x \over 2}}2\sinh({x \over 2})\prod_{n=1}^{\infty}
{{(1-q^ne^x)(1-q^ne^{-x})}\over {(1-q^n)^2}} 
\end{equation}
(here $G_2(\tau)=-{1 \over {24}}+\sum_{n \ge 1}(\sum d \vert n)q^n$ 
cf. (\ref{eisenstein}), Appendix I). Then
 $$Q(x)=e^{kx}x{{\Upsilon({x}-z,\tau)} \over 
    {\Upsilon({x})\Upsilon(-z,\tau)}}$$
(which, for $k=0$, is differ from (\ref{formula2}) by a  
factor which is $(e^{-\pi \sqrt{-1} z}\Upsilon(-z,\tau))^{{\rm dim} X}$).

Hirzebruch-Witten elliptic genus of an almost complex manifold
corresponds to the characteristic series:
\begin{equation}\label{hirzebruchwitten}
     e^{{\bar k \over N}x}{{\Upsilon(x-\alpha)} \over {\Upsilon(x)\Upsilon(-\alpha)}}  \
     \ \ \ \alpha={2 \pi \sqrt{-1}}{{(a\tau+b)}\over N}, \ \ {0 \le
       a,b, \bar k  <N}
\end{equation}
Specialization $q=0$ of (\ref{hirzebruchwitten}) yields the holomorphic
euler characteristic $\chi_y(X,K^{k \over N})$ and
specialization of  (\ref{formula1}) to $z={{2 \pi
    \sqrt{-1}(\alpha+\beta \tau)} \over N}$, up to a factor 
depending on dimension gives Hirzebruch-Witten elliptic genus 
(\ref{hirzebruchwitten}) (\cite{BLinvent}, Prop.2.4). 
As (\ref{formula2}), this is invariant of almost complex 
manifolds but it has the following modular property:
if $c_1(X)=0 \ {\rm mod} \ N$ then Hirzebruch-Witten elliptic genus 
is a modular form for the subgroup $\Gamma_0(N)$ of the modular group.
If $N=2$ in (\ref{hirzebruchwitten}) then
 the Hirzebruch-Witten genus depends on Pontryagin (rather than Chern) 
classes of $X$ only and 
is an invariant of $C^{\infty}$-manifolds which is modular (for
$\Gamma_0(2)$) if manifold is Spin. This is the first 
instance of elliptic genus which appeared in mathematics literature
and is due to Ochanine-Landweber-Stong (cf. \cite{landweber})

\subsection{Orbifold elliptic genus}\label{orbellgen} The elliptic genus of orbifolds 
which are global quotients was defined in \cite{Bduke} as follows
(this definition was extended to arbitrary orbifolds in \cite{liuma}). 

Let $X$ be a smooth projective variety and  let $\Gamma$ be a finite group
of its automorphisms. For an element $g \in \Gamma$ let $X^g$ denote its 
fixed point set. For a connected component $\bar X^g$ of  $X^g$ we consider the
decomposition 
into the eigenspaces of $g$ for the  restriction $TX\vert_{\bar X^g}$ of the
tangent bundle of $X$ on  $\bar X^g$. We represent each eigenvalue of $g$ acting
on this restriction,  in the form $exp (2\pi \sqrt {-1}\lambda(g))$
 where $0 \le \lambda(g)<1$ and denote the eigenbundle corresponding 
to this eigenvalue as $V_{\lambda(g)}$. In particular $V_0$ is the
tangent bundle to $\bar X^g$ and $TX\vert_{\bar X^g}=V_0 \oplus
(\oplus_{\lambda(g) \ne 0} V_{\lambda(g)})$.
 We also denote by $F(g,\bar X^g)=\sum \lambda(g)$,  ``the
fermionic shift'' corresponding to the component $\bar X^g$.  
Then we let
\begin{equation}
 V_{h,\bar X^h \subset X}:=\otimes_{k≥1} (\Lambda_{yq^{k-1}}^{\bullet}V_0^*
\otimes \Lambda^{\bullet}_{y^{-1}q^k}V_0\otimes Sym^{\bullet}_{q^k}V_0^*
\otimes Sym_{q^k}^{\bullet}V_0 \otimes
\end{equation}
$$\otimes[\otimes_{\lambda(h) \ne
  0}\Lambda_{yq^{k-1+\lambda(h)}}^{\bullet}V_{\lambda(h)}^* \otimes
\Lambda^{\bullet}_{y^{-1}q^{k-\lambda(h)}}V_{\lambda(h)}
\otimes Sym^{\bullet}_{q^{k-1+\lambda(h)}}V_{\lambda(h)}^*
\otimes Sym^{\bullet}_{q^{k-\lambda(h)}}V_{\lambda(h)}]
$$
With these notations one defines the orbifold elliptic genus as:
\begin{equation}\label{orbgenusdef}
   Ell_{orb}(X,\Gamma; y, q) := y^{−{\rm dim} X/2}\sum_{\{h\} \in
     Conj(\Gamma),X^h} 
  y^{F (h,X^h \subseteq X)} {1 \over {C(h)}} \sum_{g \in C(h)} L(g, V_{h,X^h
    \subseteq X} )
\end{equation}
where $Conj(\Gamma)$ is the set of conjugacy classes in $\Gamma$, 
$C(h) \subseteq \Gamma$ is the centralizer of $h \in \Gamma$ and 
$L(g,V)$ is the holomorphic Lefschetz number of $g$ with coefficients 
in a holomorphic $g$-bundle $V$ i.e. $L(g,V)=\sum (-1)^itr(g,H^i(X^g,V))$.
Equivalent form of (\ref{orbgenusdef}) is 
\begin{equation}\label{orbcommuting}
Ell(X,\Gamma; y, q) := y^{−{\rm dim} X/2}
\sum_{\{h \} \in Conj(\Gamma),X^h \subseteq X}y^{F (h,X^h\subseteq
  X)}) \chi(H^{\bullet}(V_{h,X^h\subseteq X})^{C(h)} )
\end{equation}
where $\chi(H^{\bullet}(V_{h,X^h\subseteq X})^{C(h)} )$ is the
alternating sum of the dimensions of $C(h)$-invariant 
subspaces of the cohomology of bundles $V_{h,X^h\subseteq X}$.

Atiyah-Bott holomorphic Lefschetz formula
(cf. \cite{atiyahbott}),
allows to rewrite (\ref{orbgenusdef}) as follows. For a pair of commuting 
elements $g,h \in \Gamma$, let $X^{g,h}=X^g\cap X^h$ denotes the 
set of points in $X$ fixed by both $g$ and $h$. 
Then the expression of (\ref{orbgenusdef}) in terms of
characteristic classes is:  
\begin{equation}\label{commutingpairselliptic}
 {1 \over {\vert \Gamma \vert}} \sum_{g,h,gh=hg} (\Prod_{\lambda(g)=\lambda(h)=0}x_{\lambda})
\Prod_{\lambda}
{{\theta({x_{\lambda}\over {2 \pi
             i}}+\lambda(g)-\tau \lambda(h)-z)} \over 
  {\theta({x_{\lambda}\over {2 \pi
             i}}+\lambda(g)-\tau \lambda(h))}} e^{2 \pi i z
       \lambda(h)}[X^{g,h}] 
\end{equation}
where the products are taken over all Chern roots $x_{\lambda}$
(counted with their multiplicities) 
of the eigenbundles $V_{\lambda}$  corresponding to the 
logarithms $\lambda$ of the characters of abelian subgroup of $\Gamma$
generated by $g,h$. The term in this sum corresponding to the pair in
which both $g,h$ are the identities coincides with the elliptic genus
of $X$. We shall call this term the {\it ``trivial''} sector of the orbifold
elliptic genus. It is a summand in the {\it ``untwisted''} sector
representing sum of terms corresponding to pairs $(1,h)$.  

A notable application of mathematical definition (\ref{orbgenusdef})
is the following formula for the generating function for the 
orbifold elliptic genera of symmetric products:

\begin{theorem}\label{dmvv} (cf. \cite{Bduke})  Let $X$ be a smooth projective variety
and let $Ell(X)=\sum_{m,l} c(m, l)y^lq^m$.
Then 
\begin{equation} \label{dmvvidentity}
\sum  p^nEll_{orb}(X^n,\Sigma_n;y,q) =
\Prod_{i=1}^{\infty}\Prod_{m,l} {1 \over {(1-p^iq^my^l)^{c(mi,l)}} }
 \end{equation}
\end{theorem}
The formula (\ref{dmvvidentity}) and physics proof of this identity
(\ref{dmvvidentity})  was discovered in \cite{DMVV}.

\subsection{Elliptic genus of pairs}\label{pairs} The same work
\cite{Bduke}, besides the definition of elliptic genus of global
quotients, 
contains approach to elliptic genus of singular varieties, 
based on resolution of singularities, and in which 
the assumption that singularities are the quotients is replaced 
by an assumption coming from birational geometry:

\begin{dfn}\label{defsingularities} 
 {\it $\QQ$-Gorenstein varieties with klt singularities:} 
A normal variety $X$ is called Gorenstein if a Weil $\QQ$-divisor, which 
is a multiple of the divisor of the top degree differential form, is
Cartier.
A $\QQ$-Gorenstein variety is call klt (i.e. having Kawamata
log-terminal  
singularities) if there exist a resolution of singularities $\hat X
\rightarrow X$ such that coefficients of decomposition 
$K_{\hat X}=f^*(K_X)+\sum \alpha_kE_k$ satisfy $\alpha_k>-1$.
\end{dfn}

To define the elliptic genus of singular varieties with singularities
as in \ref{defsingularities} one first defines the elliptic genus of
pairs $X,E$ where $X$ is smooth and projective and $E$ is 
a $\QQ$-divisor on $X$ i.e. $E=\sum \alpha_kE_k$ is a formal sum 
such that components $E_k$ are smooth 
divisors on $X$ intersecting transversally. 
Moreover, one assumes that $\alpha_k >-1$ for all $k$.
In this situation one defines the cohomology class, 
called the {\it elliptic class} of pair $(X,E)$: 
\footnote{in $H^*(X,\QQ)$ tensored with a ring of functions in
$z,\tau$ appearing in expanding (\ref{ellgenpairsformula}) in $x$.
The ring of quasi-Jacobi forms described in Appendix II can be used.
Often below we shall by abuse of terminology tell that we consider 
elliptic class in cohomology (or Chow groups) meaning in fact that 
this class is in the extended in such a way cohomology (or Chow theory)}

\begin{equation}\label{ellgenpairsformula}
\Ell(X,E)
=\Prod_{i} {{x_i\theta({x_i \over {2
        \pi \sqrt{-1}}}-z,\tau) \theta'(0,\tau)} \over {\theta({x_i \over {2
        \pi \sqrt{-1}}},\tau)\theta (-z,\tau)}} \Prod_k {{\theta({e_k \over {2
        \pi \sqrt{-1}}}-(\alpha_k+1)z,\tau) \theta(-z,\tau)} \over {\theta({e_k\over {2
        \pi \sqrt{-1}}} -z,\tau) \theta(-(\alpha_k+1)z,\tau)}}
\end{equation}
The elliptic genus of pair then is evaluation of the elliptic 
class on the fundamental class of $X$:
$Ell(X,E)=\Ell(X,E)[X]$. 

The fundamental property of the elliptic class of pairs, allowing to
define the elliptic genus of a singular variety as the elliptic genus of
pair consisting of a resolution and certain divisor on the latter, is
the compatibility in 
the
blowups:

\begin{theorem} 
\label{blowup}(cf. \cite{BLannals} where a more general
  statement concerning orbifold elliptic class of pairs endowed with 
a $\Gamma$-action.) Let $(X,E)$ be a pair as described above after 
Def. \ref{defsingularities}, 
$Z$ a submanifold
of $X$ transversal to irreducible (all smooth) components $E_k$
of $E=-\sum \delta_kE_k$, $f: \hat X \rightarrow X$ be the blow up of
$X$ with center at $Z$, $Exc(f)$ its exceptional divisor, $\hat E_k$
are the proper preimages of components $E_k$, $\delta$ and $\hat
E$ are such that $\hat E=-\sum \delta_k\hat E_k-\delta Exc(f)$ and 
$K_{\hat X}+\hat E=f^*(K_X+E)$. Then $\delta>-1$ and  
\begin{equation}
     f_*(\Ell(\hat X,\hat E))=\Ell(X,E)
\end{equation}
In particalar $Ell(\hat X,\hat E)=Ell(X,E)$.
\end{theorem}

\begin{corollary} Let $X$ be a $\QQ$-Gorenstein projective variety 
with at most klt singularities. Let $\hat X \rightarrow X$ be a
resolution of its singularities and $\hat E=\sum \alpha_k \hat E_k$ be 
a normal crossing divisor on $\hat X$ such that $f^*(K_X)=K_{\hat X}+\hat E$.
Then $Ell(\hat X,\hat E)$ depends only on $X$ i.e. is independent of a choice of $(\hat X,\hat
E)$ (and called ({\it singular}) elliptic genus of $X$). It  will be
denoted $Ell_{sing}(X)$.
\end{corollary}

The fundamental relation between singular and orbifold elliptic genera
is given by the so called {\it MacKay correspondence for elliptic genus}:

\begin{theorem}  Let $X$ be a smooth projective variety on 
which a group $\Gamma$ acts effectively via biholomorphic
transformations. 
Let $\mu: X \rightarrow X/\Gamma$ be the quotient map. Assume that $\mu$ does 
not have ramification divisors i.e.fixed points of elements of
$\Gamma$ have codimension greater than one.
 Then
\begin{equation}
 \Ell_{orb}(X,\Gamma; z, \tau) = ({{2  \pi \sqrt{-1} \theta(-z,\tau)}
   \over {\theta′(0, τ)}})^{\rm dim X} \Ell_{sing}(X/\Gamma, z, \tau)
 \end{equation}
In particular, orbifold elliptic genus coincides with the elliptic
genus 
of any crepant resolution (if such exist).
\end{theorem}
We refer to \cite{BLannals} theorem 5.3 for a more general statement in the
category of Kawamata log-terminal pairs and for the case of quotients
maps admitting ramification divisors.


An immediate corollary is reinterpretation of the series in theorem
\ref{dmvv} in case when $dim X=2$, 
as the generating series of the elliptic genera of Hilbert schemes:
\begin{corollary} Let $X$ be a smooth projective surface and
  $Ell(X)=\sum_{m,l}q^my^l$.
Then 
\begin{equation}
\sum  p^nEll(Hilb_n,q,y) =
\Prod_{i=1}^{\infty}\Prod_{m,l} {1 \over {(1-p^iq^my^l)^{c(mi,l)}} }
 \end{equation}     
\end{corollary}
Indeed, in the case of the surfaces the morphism $Hilb_n(X)
\rightarrow X^n/\Sigma_n$ is a smooth crepant resolution 
(i.e. $\alpha_k=0$ in definition \ref{defsingularities}).

\section{Equivariant elliptic genus}

In this section we discuss an equivariant version of the 
elliptic genus. In particular we shall 
describe equivariant analog of push forward formula (i.e. theorem 
\ref{blowup}) for 
 elliptic class, equivariant McKay correspondence,
equivariant localization and push forward properties of 
the contributions of compact components of fixed point sets into elliptic class. 
Our approach is based on equivariant intersection theory 
as developed in \cite{edidin} (cf. also \cite{totaro}). 
It allows to derive equivariant results
from their non-equivariant counterparts, already discussed in section \ref{pairs}, 
applied in appropriately formulated
context.
As in \cite{edidin} and \cite{BLannals}, instead of ordinary cohomology, we work in  
 Chow theory, but a reader of course can interpret all statements 
as those in ordinary cohomology. 

\subsection{Equivariant intersection theory}\label{general}

We start with working in the category of quasi-projective  
normal varieties (over $\CC)$ with various assumptions on
singularities such as  $\QQ$-Gorenstein
and klt conditions (cf. section \ref{pairs}). 
We also assume that a reductive algebraic group $G,\ 
{\rm dim} G=g$ 
acts on such $X$ via a linearized action. The latter means that an ample line bundle $L$ 
is presented on $X$ together with a $G$-action on the total space of $L$ such that 
bundle projection on $X$ is equivariant (cf. \cite{GIT}).
We shall refer to \cite{edidin} Section 6 for precise conditions 
on the action which assure that constructions, 
needed for equivariant intersection theory to run, will work.

Let $V, {\rm dim} V=l$ be a representation of $G$ and $U \subset V$ 
is an open set such that $G$ acts on $U$ freely and 
${\rm codim} V\setminus
U$ is sufficiently large.
Then $U/G$ is smooth and for a given $n$, the Chow groups
$A^{n'}(U/G)=A_{l-g-n'}(U/G)$ are well defined for $l>>n$,  
and so are  the products among them for all $n'<n$. 
 The Chow ring $A^*(BG)$ is defined 
as the graded ring having $A^n(U/G)$ for $n<<l$ as its graded components:
again, those are independent of $l$ as long as $l$ is large enough.

Since $G$ acts freely on $U$, the diagonal $G$-action on $X \times U$ 
is free as well, the quotient space $X_G=(X \times U)/G$ 
does exist and equivariant Chow group $A_i^G(X)$ can be defined 
as the usual Chow group $A_{i+l-g}(X_G)$.  Again, it is independent of $V,U$
as long as ${\rm codim} V\setminus U$ is sufficiently large
  (cf. \cite{edidin} Prop.-Def). The intuition behind such choice of indices is that 
in the case when $X$ is smooth, projective and the quotient 
$U/G$ is compact, one has $dim X_G={\rm dim} X+l-g$ and 
 $A_{i+l-g}(X_G)=A^{dim X-i}(X_G)$ by Poincare duality.

Let $E$ be an equivariant $G$-bundle on a quasi-projective variety with
action of $G$ i.e. the total space of $E$ is
endowed with $G$-action such that projection $E \rightarrow X$ is
$G$-equivariant. Then $E_G=(E \times U)/G \rightarrow X_G$ is a vector 
bundle on $X_G$ and equivariant Chern class $c^G_j \in A^G_*(X)$ is the Chern class
of the vector bundle $E_G$ on $X_G$. As in non-equivariant case, one
associates with an equivariant bundle the (equivariant) Chern roots $x_i^G \in A_*(X_G)$.

To define equivariant elliptic class, we note that the map $\pi$ induced by projection on the second factor: 
\begin{equation} \ \  X_G=(X \times U)/G \buildrel \pi \over
  \longrightarrow U/G=BG
\end{equation}  is 
a locally trivial fibration with the fiber $X$.
\begin{dfn}\label{equivargenusdef} Let $X$ be a {\it smooth projective}
  variety with an action of algebraic group $G$.
 The equivariant elliptic genus of $(X,G)$ is the push
  forward of the equivariant elliptic class i.e. the  class
  (\ref{formula2}) where $x_i$ are the equivariant Chern roots of the
  tangent bundle of $X$ with its natural $G$-structure:
\begin{equation}\label{equivarellgen}
  Ell^G(X)= \pi_*(\Ell(X_G)) \in A_*(BG) \otimes QJac
\end{equation} 
where $\pi: X_G  \rightarrow BG$ is induced by projection of $X \times
U$ on the second factor and $QJac$ is the ring of quasi-Jacobi forms
i.e. the ring of functions on $\CC \times H$ 
generated by coefficients of Taylor expansion in $x$ of a factor in
the product (\ref{formula2}) 
(cf. Appendix II)
\footnote{as in (\ref{ellgenpairsformula}) one can use any ring 
of functions containing the coefficients of expansion of elliptic
genera of manifolds in Chern classes.}.
\end{dfn}

By equivariant Riemann Roch theorem, one can interpret
(\ref{equivarellgen}) as the character 
decomposition of holomorphic euler characteristic of the
$G$-equivariant bundle (\ref{formula2}) where $T_X,\Omega^1$ endowed 
with natural $G$-structure (cf. \cite{edidinrr}).

In the case when $G$ is a torus $T$ (affine
connected commutative algebraic group) of dimension $r$, the equivariant 
elliptic class in $A^*(BT,QJac)$
can be viewed as an element of the ring of polynomials in $r$ variables 
with coefficients in the ring of quasi-Jacobi forms (cf. Appendix II).

\subsection{Equivariant localization.} Let $T$ be a torus 
acting algebraically on a smooth 
quasiprojective
scheme $X$. Let $\hat T$ be the group of characters of $T$. An
identification $T={\CC^*}^r$ induces the identification 
of $\hat T$ with a free abelian group generated by character $t_1,...,t_r
\in \hat T$
(such that $t_i(z_1,....,z_i,...z_r)=z_i \in \CC^*$).
$T$-equivariant Chow ring of a point, i.e. $A^*(BT)$, as was already
mentioned, is isomorphic to 
the symmetric algebra of free abelian group $\hat T$. More generally, 
if $T$ acts trivially on $X$ then $A^T_*(X)=A_*(X) \otimes Sym(\hat
T)$ (here $Sym(\hat T)$ is the symmetric algebra with generators $t_1,...,t_r$;
cf. \cite{edidinamerjourn}). For details of the following we
refer to \cite{edidinamerjourn}.
\begin{theorem}\label{localizationtheorem}
  Let $R_T=Sym(\hat T)$, $\Q_T=(R_T^+)^{-1}R_T$ where 
$R_T^+$ is the semigroup of elements of positive degree and $i: X_T
\rightarrow X$ be the embedding of the fixed point set.
Then 
\begin{equation}\label{localizationformula} 
   i_*: A_*(X_T)\otimes \Q_T \rightarrow A^T_*(X)\otimes \Q_T
\end{equation}
is an isomorphism.
\end{theorem}
 If $Y$ and $X$ are smooth, $j: Y \rightarrow
 X$ is a regular embedding of codimension $d$,  
 $N$ is the normal bundle of $Y$ in $X$ and $\alpha \in A_*(Y)$, one has 
{\it the self intersection formula} $j^*j_*(\alpha)=c_d(N) \cap \alpha$ 
(cf. Sect. 6.3, Cor. 6.3 \cite{fulton}).
If $F$ is fixed point set of a torus $T$ acting on a smooth 
scheme $X$, then $F$ is smooth and self intersection formula 
applied to $i_F: F \times U/G \rightarrow X \times U/G$ implies  
$i_F^*{i_F}_*(\alpha)=c^T_d(N) \cap \alpha$. This results in an explicite
localization isomorphism:
\begin{equation}\label{localizationformula2}
     A^T_*(X) \otimes \Q_T \rightarrow  A_*(X_T)\otimes \Q_T :\ \ \ \
     \beta \rightarrow {{i_F^*(\beta)} \over {c^T_d(N)}}\end{equation}
(here $c^T_d(N)$ denotes the equivariant Chern class of the normal
bundle to the fixed point set).

\subsection{Push forward of equivariant elliptic class and 
equivariant McKay correspondence}

Above approach to equivariant intersection theory allows to deduce
directly the equivariant counterparts of the 
key results about elliptic genus: the push forward formula of elliptic
class and the McKay correspondence. A different derivation of these
properties was given in \cite{waelder}.

Let $\bar X$ be a smooth projective variety with a biholomorphic  action of 
a torus $T$. Let $E=\sum \alpha_i E_i$ be a normal crossings divisor on $\bar X$ such that 
all irreducible components $E_i$ are $T$-invariant.  Then (in notations of section
\ref{general}) $(E_i \times U)/T$ is a divisor
on $(\bar X \times U)/T$ and hence the classes $e_i^T \in A^T_*(\bar X)$
are well defined. Using (\ref{ellgenpairsformula}) we
obtain {\it the equivariant elliptic class} $\Ell^T(\bar X,E) \in A^T_*(\bar X,QJac)$. 

\begin{theorem}\label{pushforward} ({\rm Push forward formula}.)
Let $\bar X$  be a smooth projective variety 
with a torus 
$T$ acting on $\bar X$ via biregular automorphisms. Let $E$ be a
$T$-invariant normal crossings divisor and $Z$ a smooth $T$-invariant submanifold of $\bar X$ 
transversal to all irreducible components of $E$. Let $\phi: \bar X'
\rightarrow \bar X$ be the blow up of $\bar X$ with center at $Z$ and
let $E'$ be 
the divisor on $\bar X'$ such that $\phi^*(K_{\bar X}+E)=K_{\bar
  X'+E'}$. 
Then the action of $T$ on $\bar X\setminus Z$ extends to the action 
on $\bar X'$ leaving $E'$ invariant and 
\begin{equation}\label{pushforwardfla}
   \phi_*(\Ell^T(\bar X',E'))=\Ell^T(\bar X,E)
\end{equation}
where on the left one has the equivariant elliptic class for the
action on $\bar X'$ induced by the action of $T$ on $\bar X \setminus Z$.
\end{theorem}
\begin{proof}
Let $\pi: \bar X_{T} \rightarrow BT$ be a locally trivial fibration
defined by the action of $T$ and a representation of $T$ as in section
\ref{general} (recall that $BT=U/T$ is the quotient space of a Zariski
open set $U$ in the representation space with sufficiently large codimension of
the complement to $U$). 
Since $Z$ and $E_i$ are $T$-invariant, one has embedding
of fibrations $Z_{T} \rightarrow \bar X_{T}, (E_i)_T \rightarrow
\bar X_T$ of subvarieties of $\bar X_T$ corresponding to $Z$ and $E_i$
compatible with projections 
on $T$.  
Let $\bar X'_{T}=(\bar X' \times U)/T$ and $\phi_{T}: \bar X'_{T}
\rightarrow \bar X_{T}$ be induced morphism. 
$\bar X'_T=(\bar X'\times U)/T$ can be identified with  the blow up 
of $\bar X_T$ along $Z_T$. This can be seen for example from a local
description of blow up as in \cite{voisin} Def. 3.23. 
Moreover,  $E_T=\sum
\alpha_i{(E_i)}_T$, 
the multiplicity of ${(E_i)}_T$ along $Z_T$ is the same as multiplicity
$\beta_i$ of $E_i$ along $Z$ and codimension of $Z_T$ in $\bar X_T$
coincides with the codimension of $Z$ in $\bar X$. It follows that 
${(E'_i)}_T$, which irreducible components are the proper preimages of
${(E_i)}_T$, and the exceptional locus of $\phi_T$ all have the same 
multiplicities as do the corresponding components in $E'$ (cf. \cite{Bduke}
 p.327 and also theorem \ref{blowup}). 
Therefore $\phi_T^*(K_{\bar X_T}+E_T)=K_{\bar X_T'}+E'_T$. 
Now Theorem 3.5 in \cite{BLannals} immediately
implies Theorem \ref{pushforward}. 
 \end{proof}

As in non-equivariant case, push forward formula (\ref{pushforwardfla})
shows that the following definition is independent of resolution it
uses.

\begin{dfn}{\it Equivariant singular elliptic class.}  Let $X$ be a $\QQ$-Gorenstein projective variety 
with at most klt singularities on which a torus $T$ acts by
regular automorphisms. Let $f: \hat X \rightarrow X$ be an 
equivariant resolution of its singularities and $\hat E=\sum \alpha_k \hat E_k$ be 
a normal crossing divisor on $\hat X$ such that $f^*(K_X)=K_{\hat X}+\hat E$.
Equivariant singular elliptic class is defined as 
\begin{equation} \Ell^T_{sing}(X)=f_*(\Ell^T(\hat X,\hat E))
\end{equation}
(it is independent of a choice of equivariant resolution).
Equivariant singular elliptic genus is the push forward of
$\Ell^T_{sing}(X)$  to the Chow ring of the point (cf. Def. \ref{equivarellgen}).
\end{dfn}

In the case when the singular variety is an orbifold with an action of a
torus one has equivariant version of orbifold elliptic class related
to just described equivariant singular elliptic class.

\begin{theorem}\label{mckay}
 ({\rm Equivariant version of McKay correspondence}) 
Let $X$ be a smooth projective variety
with a torus 
$T$ acting on $X$ via biregular automorphisms. Let 
$\Gamma$ be a finite group which action 
commute with the action of $T$. 
Then, for any pair $(g,h) \in \Gamma$ of commuting elements,  
the fixed point locus $X^{g,h}$ is $T$-invariant,   
the class obtained by replacing in elliptic class appearing in
(\ref{orbcommuting} )
the ordinary Chern roots of the bundles $V_{\lambda}$ by the 
equivariant Chern roots of these bundle with natural $T$-structure, and
called the {\bf equivariant orbifold class} of $(X,T,\Gamma)$, satisfies the
following push forward formula. \footnote{Here we consider the full 
elliptic class i.e. for each commuting pair $g,h$ one takes 
the cap product of class obtained by expansion of
$\theta$-functions with the fundamental class  $[X^{g,h}]$. 
This cup product is an element of equivarinat Chow ring of $X^{g,h}$.
The push forward of this cap product to the Chow ring of a point gives the equivariant
orbifold elliptic genus and is an element in the ring of formal power series in characters
of $T$.}
If $\psi: X \rightarrow
X/\Gamma$ is the quotient morphism, then 
\begin{equation}
       \psi_*(\Ell^T_{orb}(X,\Gamma))=\Ell^T_{sing}(X/\Gamma)
\end{equation}
\end{theorem}
\begin{proof}  This  follows from corresponding results in
  \cite{BLannals} as in the proof of theorem \ref{pushforward}.
Since the actions of $T$ and $\Gamma$ commute, the torus $T$ acts 
on $X^{\gamma},\gamma \in \Gamma$, the action of $\Gamma$ on $X$
 induces the action of $X_T=(X \times U)/T$  via action on the first
 factor and the fixed point set of $\gamma \in \Gamma$ is
 $X^{\gamma}_T$. 
Hence $\Ell^T_{orb}(X,\Gamma)=\Ell_{orb}(X_T,\Gamma)$. 
Now the theorem follows from the theorem 5.3 in \cite{BLannals} 
applied to the action of $\Gamma$ on $X_T$.

\end{proof}

\subsection{Push forward of contributions of components of fixed point
  set.}

The localization map (\ref{localizationformula}) allows to associate with 
a fixed component $F$ of an action of a torus an invariant 
constructed using contribution of $F$ into 
equivariant  elliptic class of $X$. In the case when $X$ is a smooth
projective variety the sum over all fixed components 
of these contributions evaluated on corresponding fundamental classes of the
components 
coincides with the equivariant elliptic genus of $X$ (cf. \cite{atiyahbott}).
In the case when $X$ is only {\it quasi}-projective but  a component $F$ 
is compact, the corresponding contribution is well defined and though
 by itself it does not have a geometric interpretation,
this contribution  does play the key
role in definitions of next section. Here we shall describe the push
forward property  of  contributions of compact components 
and its generalization to 
the orbifold case.

\begin{dfn} ({\it Local contribution of a component of fixed
    point set: smooth case}.) Let $X$ be a smooth quasi-projective variety, $T$ as
  above and let $E$ denotes a normal crossing divisor with $T$-invariant 
irreducible components.
Let $F$ be a component 
of the fixed point set. Assume that $F$ is compact 
and let $i_F: F \rightarrow X$ denotes its embedding. 
Let $c^T_{{\rm codim} F}(N_{F})$ be the equivariant Chern class of
the normal bundle of $F$ in $X$. 
Then the local contribution of $F$ into equivariant elliptic genus of
the pair $(X,E)$ is the class 
\footnote{the ring in this formula can be taken to be $A_*(F,QJac) \otimes \Q$}:
\begin{equation} \Ell^T_F(X,E)={i_F^*{\Ell(X,E)}\over {c^T_{{\rm codim} F}(N_F)}}\in
  A_*(F)\{\{q,y\}\}\otimes \Q
\end{equation}
\end{dfn}

\begin{theorem}\label{pushforwardlocas}
({\rm Push forward for local contribution of equivariant elliptic
  genus}) 
Let $X$ be a smooth quasi-projective
variety with action of a torus $T$ and let 
$F \subset X$ be a component of the fixed
point set which is compact. Denote by  $\phi: X'\rightarrow X$  
$T$-equivariant blow up with $T$-invariant center $Z \subset F$  
and let  $\F'_F=\bigcup_{F' \in Irr\F'}F'$  be the union of submanifolds $F'$ from the set $Irr(\F')$ of 
irreducible components of the fixed point set $\F'$ of the action
of $T$ on $X'$ mapped by $\phi$ onto $F$.
Let $E$ be $T$-invariant normal crossing divisor 
all component of which are 
transversal to $Z$ and $E'$ be the divisor on $X'$ 
 such that $\phi^*(K_X+E)=K_{X'}+E'$.
Then 
\begin{equation}\label{pushforwardlocal}
\phi_*\sum_{F'\in Irr (\F')}{i^*_{F'}{\Ell^T(X',E')}\over {c^T_{Codim
      F' \subset X'}(N_{F/X'})}}={{i^*_{F}{\Ell^T(X,E)}}\over {c^T_{Codim
      F\subset X}(N_{F/X})}}
\end{equation}

\end{theorem}

\begin{proof} Let $\bar X$ be a compactification of $X$ 
and $\bar X'$ be the blow up of $\bar X$ at $Z \subset X \subset \bar
X$. Let $\F'=\bigcup_{F' \in Irr(\F')} F'$ (resp. $\F=\bigcup_{F \in
  Irr \F} F$) be the submanifold of $\bar X'$ 
of fixed points of action of $T$ on $\bar X'$ (resp. $\bar X$) and
$i_{\F'}: {\F'}\rightarrow \bar X'$ 
(resp. $i_{{\F}}:{\F} \rightarrow \bar X$)
be their embeddings. 
The push forward formula of theorem \ref{pushforward} can be rewritten as:
\begin{equation}
 {i_{\F}}_*({i_{\F}}_*^{-1}\phi_*
 {i_{\F'}}_*){{i^{-1}_{\F'}}_*{\Ell^T(\bar X',E')}}=
{{{\Ell^T(X,E)}}}
\end{equation}
Now using description of the inverse of $i_*$ given in (\ref{localizationformula2})
and $(\phi\vert_{\F'})_*={i_{\F}}_*^{-1}\phi_* {i_{\F'}}_*$ we obtain
\begin{equation}
 \phi_*{i^*_{\F'}{\Ell^T(\bar X',E')}\over {c^T_{top}(N_{\F'/\bar
       X'})}}={{i^*_{\F}{\Ell^T(X,E)}}\over 
{c^T_{top}(N_{\F/\bar X})}}
\end{equation}

Fixed point set $\F'$ (resp. $\F$)  is a disjoint union of smooth irreducible components
and hence $A^*(\F')=\oplus_{F'\in Irr(\F')} A^*(F')$ (similar direct
sum decomposition for
$\F$) where summation is over  
the set $Irr(\F')$ of irreducible components of $\F'$ (resp. $\F$).
The split is given by projections $i_{F'}^*: A^*(\F') \rightarrow
A^*(F')$ 
(resp. $i_{F}^*: A^*(\F) \rightarrow
A^*(F)$) where $i_{F'}: F'\rightarrow \F'$ is embedding of an
irreducible component into the disjoint union  (and the same for $F$).
The map ${\phi\vert_{\F'}}_*$ respects the above direct sum
decomposition 
with ${\phi\vert_{\F'}}_*^{-1}(A^*(F))=\oplus_{F' \in Irr(\F')}
A^*(F')$.
This implies (\ref{pushforwardlocal}).
\end{proof}

\subsection{Contributions of components of fixed point set into 
orbifold elliptic genus} 
 Let $X$ be a smooth quasi-projective variety, let $T$ be  a torus
  acting on $X$ effectively and let $\Gamma$ is a finite group acting upon $X$, 
(all actions are via biholomorphic automorphisms).
We shall assume that 
the action of $\Gamma$ commutes with the action of $T$ 
i.e. for all $t \in T, \gamma \in \Gamma$ and any $x \in X$ 
one has $\gamma t \cdot x=t \cdot \gamma x, \gamma, t \in Aut(X)$.
This implies that $\Gamma$ leaves invariant the fixed point set $X^T$ of 
the torus $T$, each fixed point set $X^g, g \in \Gamma$ is $T$-invariant
 and that $T$ acts on the quotient $X/\Gamma$. We denote 
by $T^{eff}$ the quotient of $T$ which acts {\it effectively} on $X/\Gamma$.

If $F$ is a connected component of $X^T$ and 
$F^{\gamma}$ is a component of the fixed point set of an element 
 $\gamma \in \Gamma$ acting upon $F$ 
then restriction of cotangent (or tangent) bundle $\Omega^1_X$ of $X$ 
on $F^{\gamma}$ has a canonical structure of an equivariant
$T$-bundle. 
If $V$ is an eigenbundle of this $T$-action on
$\Omega^1_X\vert_{F^{\gamma}}$ then, since we assume that actions 
of $\Gamma$ and $T$ commute,  
$V$ is invariant under the action of $\gamma$ as well.

If $rk V=1$ then, as in section \ref{orbellgen}, for
$\gamma \in \Gamma$ we let $\lambda(\gamma)$ 
denote the logarithm ${1 \over {2 \pi \sqrt{-1}}}log \in [0,1)$ of the
value on $\gamma$ of the character of action on $V$ of the 
subgroup $<\gamma>$ of $\Gamma$ generated by
$\gamma$.  We assign the subscript $\lambda$ to such a line
bundle $V$, put $x_{\lambda}=c^T_1(V_{\lambda}) \in A^2_T(F^{\gamma})$
and count the class $x_{\lambda}$ with multiplicity equal to the multiplicity of the
character $\gamma \rightarrow exp(2 \pi \sqrt{-1}\lambda(\gamma))$
in the bundle $\Omega^1_X\vert_{F^{\gamma}}$. Similar collection 
of equivariant Chern classes arises from the normal 
bundles to the fixed point sets $F^g\cap F^h$ of pairs $g,h$
commuting elements in $\Gamma$.
\begin{dfn}\label{defPhi}
Let $F \subset X$ be a connected {\it compact} component of the fixed
point set of an action of 
$T$ and let for a commuting pair $g,h \in \Gamma$,  $F^{g,h}$ denotes submanifold of $F$
consisting of the points fixed by both $g$ and $h$. 
We associate with a connected component of $F^{g,h}$ and a rank one 
$T$-eigenbundle $V$ of $\Omega^1_X\vert_{F^{g,h}}$, 
 the characteristic class in the ring $A^*_T(X,\CC)[[q,y]]$ given by: 
\begin{equation}\label{equivariantorbifoldsummand}
     \Phi^T_{F^{g,h}}(x^T,g,h,z,\tau,\Gamma)={{\theta({x^T\over {2 \pi
             i}}+\lambda(g)-\tau \lambda(h)-z)} \over 
  {\theta({x^T\over {2 \pi
             i}}+\lambda(g)-\tau \lambda(h))}} e^{2 \pi i z
       \lambda(h)} 
\end{equation}
where $x^T$ is equivariant Chern class of $V$. 
\end{dfn} 
Below we also denote 
by $Cong(\Gamma)$ the set of conjugacy classes of $\Gamma$, 
  $C(g), g \in \Gamma$ will denote the centralizer of $g$, 
$\Lambda$ be the set of $(g,h)$-eigenbundles of tangent bundle to $X$ 
restricted to $F^{g,h}$ 
and $\Lambda_{F^{g,h}}$ will be the  
collection of  $(g,h)$-eigenbundles of $N_{F^{g,h}\subset F}$ 
such that $\lambda(g)=\lambda(h)=0$.

\begin{dfn}\label{equivariantorbifoldcontribution}
The contribution of $F \in X^T$ into $T$-equivariant orbifold elliptic genus
of $(X,\Gamma)$ is the sum:
\begin{equation}\label{equationequivariantorbifoldcontribution}
\Ell^{T^{eff}}_F(X,\Gamma,u,z,\tau)=
\end{equation}
$$\sum_{\{g\} \in Cong(\Gamma)} {1 \over {\vert C(g) \vert}} \sum_{h
    \in C(g)}  
(\prod_{\lambda \in \Lambda_{F^{g,h}}}
x_{\lambda})
\prod_{\lambda \in \Lambda} 
\Phi^{T^{eff}}_{F^{g,h}}(x_{\lambda},g,h,z,\tau,\Gamma)[F^{g,h}]
$$
where $N_{F^{g,h} \subset F}$ is the normal bundle to $F^{g,h}$ in
$F$ and all equivariant
Chern classes expressed in terms of the characters of $T^{eff}$.
\end{dfn}

The motivation of this definition is the following. 
Orbifold elliptic genus (\ref{commutingpairselliptic})
is a sum over pairs of commuting elements in 
$\Gamma$
of  classes in the Chow ring (which are combinations of Chern
classes $x_{\lambda}$ of bundles $V_{\lambda}$ in (\ref{orbgenusdef})) evaluated on the
fundamental class of $X^{g,h}$ (cf. proof of theorem 4.3 in
\cite{Bduke}). In the case when $X$ 
is projective, the localization formula (cf.(\ref{localizationformula2}))
applied to the equivariant version of the orbifold elliptic genus replaces
each summand in (\ref{commutingpairselliptic})
by the sum over components $F$ of $X^T$
of pullbacks to $F^{g,h}=F \cap X^{g,h}$ classes  (\ref{commutingpairselliptic}) 
 divided by the equivariant top Chern class 
of the normal bundle to $F^{g,h}$ in $X^{g,h}$.
Definition \ref{equivariantorbifoldcontribution} is the sum over
$(g,h)$ of contribution from one individual component $F$.

\begin{example}{\it Trivial sector of contribution described in
    Definition \ref{equivariantorbifoldcontribution} for $\chi_y$-genus.} 
 Specialization to the case $q=0$ of the term corresponding to pair $g=h=1$
(i.e. the trivial,sector) gives the following local contribution of component $F$ of fixed point set of
action of $T$ on $X$ into $\Gamma$-orbifold
$\chi_y$-genus:
\begin{equation}\label{chiycontribution}
              {\chi_y(X,\Gamma,g=h=1)}^{T^{eff}}_F=
\Prod_{\lambda} (y^{-{1 \over 2}}x_{\lambda(T)}{{1-ye^{-x_{\lambda(T)}}}\over
{1-e^{-x_{\lambda(T)}}}}) \cdot  
\Prod_{\lambda} (y^{-{1 \over 2}}{{1-ye^{-x^n_{\lambda(N)}}}\over
{1-e^{-x^n_{\lambda(N)}}}})
\end{equation}
where $x_{\lambda(T)}$ are the Chern roots of the tangent bundle to $F$ 
(appearing in the first product) 
and $x^n_{\lambda(N)}$ are the equivariant Chern roots of the normal
bundle to $F$ (contributing to the second factor in
(\ref{chiycontribution})).
Indeed, in the sector $g=h=1$ in
(\ref{equationequivariantorbifoldcontribution}), 
 we have only one term which is
specialzation of class $\Phi$ given in Definition \ref{defPhi}.
\end{example} 

The contributions into orbifold elliptic genus corresponding to
compact components of fixed point set described in 
Def. \ref{equivariantorbifoldcontribution}
satisfy the following, localized at $F$,  McKay correspondence proof of which can 
be obtained in the same way as the proof of theorem \ref{pushforwardlocas}.
More general case, providing local equivariant version for pairs 
as in \cite{BLannals} can be obtained similarly.

\begin{theorem}\label{localequivarmckay} Let $X$ be a smooth quasi-projective variety, $T$ 
a torus and $\Gamma$ a finite group both acting on $X$ via biholomorphic 
automorphisms so that their actions commute i.e. 
$\gamma \cdot t v=t 
\cdot \gamma v, \gamma \in \Gamma, t \in T, v \in X$. 
Let $\phi: \tilde X \rightarrow X/\Gamma$ be a crepant
resolution of singularities of the quotient $X/\Gamma$ (if it exist.) As above,
denote by  $T^{eff}$ the quotient of $T$ by the finite group which acts {\it effectively}
on $X/\Gamma$. 
Let $F \subset X/\Gamma$ be a component of the fixed point set of
$T^{eff}$ and $\F'$ be the collection of components of the fixed point 
set of $T$ such that $\phi(F')\subset  F, F' \in \F'$. Then
\begin{equation}
     \sum_{F'\in \F'} \Ell^{T^{eff}}_{F'}(\tilde X) =\Ell^{T^{eff}}_F(X,\Gamma)
\end{equation}
\end{theorem}

In the next section, we consider explicite examples of  calculations
of contributions of fixed components 
of $\CC^*$-actions on the GIT quotients by the actions of tori on
 bundles over quasi-projective varieties. They will provide 
ample illustration to the theorem \ref{localequivarmckay}.

\section{Elliptic genus of phases.}

This section discusses applications of the local contributions of compact 
components of the fixed point sets introduced in previous section
in the special case when action of $T=\CC^*$ takes place 
on a GIT quotient of the total 
space of a vector bundle by an action of reductive group. 
This action of $T$ is canonical in the sense that it is induced from the action 
of $T$ on the total space of vector bundle
by dilations $v \rightarrow t \cdot v,
t \in \CC^*$. This is an extension of the framework of examples considered
by Witten in \cite{n2}. Following this work,  in \cite{phasesme} we
called our GIT
quotient {\it phases} we well. We also attached to such framework an 
elliptic genus 
and describe its orbifoldization when additional symmetries are present. 
We show that this extends well known elliptic genera of Landau Ginzburg
and $\sigma$-models.

\subsection{Phases.}
  
We will start with a very special
example of a phase considered by Witten (\cite{n2}) 
in which we calculate contribution of  component of
the fixed point set not into elliptic genus
but rather into $\chi_y$-genus (which is 
the limit $q \rightarrow 0$ of the  elliptic genus). ``Advantage'' of $\chi_y$-genus
of course is that this is a Laurent polynomial, rather than a more
general holomorphic
function. In this example we work with $\chi_y$-genus directly,
i.e. perform localization of $\chi_y$-genus rather than elliptic
genus. Already this calculation in the case of Landau-Ginzburg phase 
results in Arnold-Steenbrink spectrum of weighted homogeneous singularity, 
providing interpretation of the latter using equivariant cohomology.

\begin{example}\label{lgexample} Let $w_1,...,w_n,D$ be collection of positive integers. 
Consider $G=\CC^*$-action  on $\CC \times \CC^n$ given by 
\begin{equation}\label{weightedaction}
      \lambda(s,z_1,....,z_n)=(\lambda^{-D}s,\lambda^{w_1}z_1,....,\lambda^{w_n}z_n)
\end{equation}
The quotient of the subset in $\CC\times \CC^n$ given by $s\ne 0$ is 
the orbifold $W/\mu_D$ where $W$ is a vector space, ${\rm dim}_{\CC}W=n$
and $\mu_D$ is the group of roots of
unity of degree $D$ acting via multiplication on $W$.
The group $T^{eff}=\CC^*$ acts on $\CC \times \CC^n$ via $t(s,z_1,...,z_n) 
\rightarrow (ts,z_1,...,z_n)$ and this action induces effective acton 
of $T^{eff}$ on $W/\mu_D$. The effective $\CC^*$-action on $W$, which induces this
action of $T^{eff}$ on $W/\mu_D$, is multiplication by $r \in \CC^*$ where $r^{-D}=t$ 
($r$ and $rexp({{2 \pi \sqrt{-1} l} \over D}), l \in \ZZ$ induce the same
automorphisms 
of $W/\mu_D$) i.e. $W$ is acted upon by $T=\CC^*$ which the
$D$-fold cover of the group $T^{eff}$ acting on $W/\mu_D$. In particular 
infinitesimal characters of the normal bundle at the fixed point $\O$  
of the action of  $T$ on $W$ (i.e. the origin) in terms of the characters
of $T^{eff}$ are $x_i=-{w_iu \over D}$, where $u$ is the infinitesimal character
of $T$.  

It follows from (\ref{chiycontribution}) that the trivial sector of the 
local contribution of $\O$ into orbifold $\chi_y$-genus is 
given by:
\begin{equation}
\Prod_i y^{-{1 \over 2}}({{w_iu} \over D } {{1-ye^{-{{w_iu} \over D}}} \over
{1-e^{{-w_iu}\over D}}}) \cdot {1 \over {{{w_iu} \over D }}}
\end{equation}
For special value of $u$ given by $u=2 \pi \sqrt{-1} z$ one obtains:
\begin{equation} 
\Prod_iy^{-1 \over 2}{{1-y^{(1-{{w_i} \over D})}} \over
{1-y^{{-w_i}\over D}}}=\Prod_iy^{-1 \over 2} {{y^{w_i \over D}-y} \over
 {{y^{w_i \over D}-1}}}
\end{equation} 
\end{example}
\noindent which coincides with  generating function of spectrum 
as calculated in \cite{st} (its definition reminded in Prop. \ref{reductiontoy}).
\smallskip

Now we consider general case for which Example \ref{lgexample} is an illustration.

\begin{dfn}\label{defphases}(cf. \cite{phasesme})  Let $\E$ be the
  total space of a vector bundle $E$ on a
  smooth quasi-projective manifold $X$. Let $G$  be a reductive 
algebraic group acting by biholomorphic authomorphism on $\E$. Let $\kappa$ be a
linearization 
of this $G$-action 
satisfying the conditions of Prop.3.1 in \cite{phasesme} 
\footnote{which implies that the   
$\CC^*$-action  by dilations is well defined on the GIT quotient}.
Phase of $G$-action on $\E$ corresponding to linearization
$\kappa$ 
is the GIT quotient $\E//_{\kappa}G=\E^{ss}/G$ endowed with the
$\CC^*$-action induced by 
$\CC^*$-action given by dilations $t(v)=t \cdot v,  t \in
\CC^*, v \in \E$.

A phase is called Landau-Ginzburg if this GIT quotient is
an orbifold biholomorphic to a quotient of $\CC^n$ by a finite
subgroup of $GL_n(\CC)$.

A phase is called $\sigma$-model (resp. Calabi Yau) if this GIT quotient is biholomorphic to
the total space of a vector bundle (resp. the canonical bundle) over a compact orbifold.
\end{dfn}

Change of linearization $\kappa$ of $G$-action on $\E$ may result 
in a birationally equivalent GIT-quotient. More specifically, 
if $NS^G(\E)$ denotes the equivariant Neron-Severi group 
(in the case when $\E$ is an affine space  this is just the group 
${\rm Char} G$ of characters of $G$), then 
there is a
partition of $NS^G(\E) \otimes \QQ$ into a union of cones such that GIT-quotients are
biregular for linearizations within a cone and $\E//_{\kappa}G$ acquires change when 
$\kappa$ belongs to the boundary of a cone or is moving into 
adjacent one. For general discussion of changes of GIT-quotients
we refer to \cite{thaddeus} or \cite{dolgachev} and to \cite{phasesme}
for particular 
case of the total spaces of bundles  as in Definition \ref{defphases}.

\smallskip

GIT-quotients are often singular but we will be interested in the cases when they are
biholomorphic to global quotients of a smooth manifold which 
we call {\it uniformization} of a global quotient. 

\begin{dfn}\label{uniformization} A smooth quasi-projective variety $\bar X$ together 
with an action of a finite group $\Gamma$ is called an 
uniformization of a phase 
  $\E//_{\kappa}G$ if 

1. there exist 
a biholomorphic isomorphism $\E//_{\kappa}G \rightarrow \bar X/\Gamma$

2. there is an action of 1-dimensional complex torus $T$ 
on $X$, a finite degree covering map $\pi: T  \rightarrow T^{eff}$ of
1-dimensional torus $T^{eff}$ acting on $\E//_{\kappa}G$ via dilations
(cf. Def. \ref{defphases}) 
such that the quotient map $\phi: \bar X \rightarrow
\bar  X/\Gamma=\E//_{\kappa}G$ is
equivaraint i.e. $\phi(t \cdot x)= \pi(t) \cdot \phi(x), t \in T$.
\end{dfn}

The following is an illustration to Definitions \ref{defphases} and
\ref{uniformization} 
with example borrowed from \cite{n2}.

\begin{example}\label{lgexample2}
 Quotient in Example \ref{lgexample} is a special case of
  the quotients considered in Definition 
\ref{defphases} with 
$X=\CC^n$, $\E=\CC^{n+1}$ being the total space of
  the trivial line bundle and  
$G=\CC^*$ acting on $\E$ via (\ref{weightedaction}).
In this case $Char(\CC^*) \otimes \QQ=\QQ$, there are 
two cones and for a pair of linearizations $\kappa_1,\kappa_2$ 
from distinct cones, the corresponding semi-stable loci are:
\begin{equation}
      {(\CC^{n+1})}^{ss}_{\kappa_1}=\CC \times (\CC^n-0)  \subset \CC^{n+1}\ \ \
      {(\CC^{n+1})}^{ss}_{\kappa_2}=(\CC-0) \times \CC^n \subset \CC^{n+1}
\end{equation}
In the simplest case, when $w_i=1,D=n$,
the corresponding GIT quotients are respectively the total space $[\O_{\PP^{n-1}}(-n)]$ of
the canonical bundle over $\PP^{n-1}$ and the quotient $W/\mu_n, W=\CC^n$ 
by the group of roots of unity of degree $n$ acting diagonally.  
As was mentioned in discussion of Example \ref{lgexample}, 
the dilations $t\cdot (s,z_1,...,z_n)=(ts,z_1,...,z_n)$ 
 induce on $\CC^n/\mu_n$ the action 
$t \cdot [(z_1,...,z_n) \ {\rm mod } \mu_n] =t(1,z_1,..,z_n) \ {\rm
  mod} \ \CC^*=
(t,z_1,...,z_n) \ {\rm mod} \ \CC^* =(t^{-{1 \over n}}z_1,...,t^{-{1 \over n}}z_n) {\rm
  mod} \mu_n$. 
This action is effective on the quotient $W/\mu_n$.
Denote by $\pi: \lambda \rightarrow t=\lambda^{-n}$ the (cyclic) covering map 
of one-dimensional tori $\CC_{\lambda}^*\rightarrow \CC_t^*$ and 
let $\phi: \CC^n \rightarrow \CC^n/\mu_n$ be
the quotient map.    Assume that  $\CC^*_{\lambda}$ is acting on $\CC^n$ via
multiplication of coordinates by $\lambda$ and $\CC^*_t$ acts on
$\CC^n/\mu_n$ as above.
Then $\phi(\lambda v)= t \cdot \phi(v)$ and therefore we have an 
uniformization 
 in the sense of Definition \ref{uniformization}. Hene we have a LG
 phase.
The quotient which is the total space $\O_{\PP^{n-1}}(-n)$ is the
$\sigma$-model (in fact CY) 
phase. Here GIT quotient is smooth, dilations on $\CC \times \CC^n$
induce on $\O_{\PP^{n-1}}(-n)$ the multiplication 
by elements of $\CC^*$ which is an effective
action and does not require 
uniformization. 
\end{example}


\subsection{Elliptic genus of a phase}

Next we shall define elliptic genus of a phase for which 
the fixed point set of  $\CC^*$-action induced by dilations
has a compact component.

\begin{dfn}\label{defellgenphase} ({\it Elliptic genus of a phase}) Let 
$X,G,\E,\kappa$ 
be as in Def. \ref{defphases}. 
Assume that $\E//_{\kappa}G$ admits uniformization
$\widetilde{\E//_{\kappa}G}$ i.e. $\widetilde{\E//_{\kappa}G}/\Gamma=
{\E//_{\kappa}G}$ for an action of a finite group $\Gamma$ and that 
one has the action of $T=\CC^*$ on $\widetilde{\E//_{\kappa}G}$ 
such that the quotient map $\widetilde{\E//_{\kappa}G} \rightarrow {\E//_{\kappa}G}$ 
 is equivariant for
the $\CC^*$-action on $\E//_{\kappa}G$
induced by dilations on $\E$.
Let $F \subset \widetilde{\E//_{\kappa}G}$ be a 
compact component of  fixed point set of $T$-action   
on $\widetilde {\E//_{\kappa}G}$. 
Consider the local contribution of the component $F$ 
into $T$-equivariant orbifold elliptic genus  
\begin{equation}\label{lasteq}
Ell^{T^{eff}}_{orb}(\widetilde {\E//_{\kappa}G},\Gamma,u,z,\tau)
\end{equation} 
given by (\ref{equationequivariantorbifoldcontribution}) in Definition
\ref{equivariantorbifoldcontribution}
where $u$ is an infinitesimal character of the action of  maximal,
effectively acting quotient $T^{eff}$.
Then the elliptic genus of the phase $(X,G,\E,\kappa)$ 
relative to the component $F$, denoted as
$Ell(\E//_{\kappa}G,F,z,\tau)$, is defined {\it as the 
restriction of the local contribution (\ref{lasteq})
on the diagonal $u=z$} of $\CC\times \CC\times \H$:
\begin{equation}Ell(\E_{\kappa}//G,F,z,\tau)=
Ell^{T^{eff}}_{orb}(\widetilde {\E//_{\kappa}G},\Gamma,F,z,z,\tau)
\end{equation}

More generally, the same definition can be used in the 
cases when $\widetilde {\E//_{\kappa}G}$ has Kawamata 
log-terminal singularities and when $Ell(\widetilde {\E//_{\kappa}G},\Gamma)$
is well defined as the orbifold elliptic 
genus of a pair obtained via a resolution of singularities and taking 
into account the divisor determined by the discrepancies of the resolution
(cf. \cite{BLannals}).
\end{dfn}

In the next theorem we shall describe a class of phase transitions 
in which one can apply equivariant McKay correspondence to obtain 
invariance of elliptic genus in such transitions.  

\begin{theorem}\label{invariancephasetransition}
 {(\rm Invariance of elliptic genus in phase transitions}.)
 Let $\E//_{\kappa_1}G=X_1$ $=\bar X_1/\Gamma$, $\E//_{\kappa_2}G=                    
X_2$ $=\bar X_2/\Gamma,                                                            
\tilde X_1,\tilde X_2,\Gamma$ are as
in \ref{uniformization}. Assume that
$\psi: X_1 \rightarrow X_2$ is
a K-equivalence i.e. $\psi^*(K_{X_2})=K_{X_1}$.
Then
\begin{equation}\label{maintheoremformula}
  \sum_{P_i}Ell(\L//_{\kappa_1},F_i)=
 Ell(\L//_{\kappa_2},F)
\end{equation}
where $F_i$ is collection of fixed point sets which $\psi$ takes
into $F$.
\end{theorem}

\subsection {Quotients of phases by the action of a finite group.}\label{quotientsphasessection}

 Constructions of mirror symmetry in toric or weighted homogeneous 
 case  (cf.  \cite{batyrev} and \cite{berglund})
suggest to consider orbifoldization of phases with respect to finite
groups. Even the very first construction of mirror symmetric of Calabi Yau
quintic in $\PP^4$ (cf. \cite{candelas}) was obtained via
orbifoldization. 
The orbifoldization of elliptic genus of Calabi Yau and Landau-Ginzburg models was
proposed in \cite{berglund}, \cite{kawaiyang}. Here we discuss
orbifoldization of arbitrary phases including hybrid ones.

Let $X$ be quasi-projective manifold with an action 
for a reductive group $G$ and let $\Gamma$ be a finite 
subgroup of the group of biregular automorphisms
of $X$ which normalizes $G$ i.e. for any $\gamma \in \Gamma$, 
$g\in G$ one has $\gamma g \gamma^{-1} \in G$. 
We say that $\Gamma$ {\it normalizes a linearization $\kappa$} of
$G$-action on $X$ 
if  action of $\Gamma$ of $X$ lifts to the action on the total space
of the ample line bundle $L_{\kappa}$ underlying $\kappa$ so that this lift
normalizes the action of $G$ on the total space of $L_{\kappa}$.
This assumption implies that $\Gamma$ acts on the semi-stable locus
\begin{equation}
     X^{ss}=\{ x\in X \vert \exists s \in \Gamma(X,L_{\kappa}^{\otimes m})^G,
     s(x) \ne 0\}
\end{equation}
Here the action of either $G$ or $\Gamma$ on $\Gamma(X,L_{\kappa}^{\otimes m})$ is given by 
$(gs)(x)=gs(g^{-1}x)$ ($g$ is an element of either $G$ or $\Gamma$). 
Indeed, $(\gamma \cdot s)(\gamma(x))=
\gamma s(x) \ne 0$ if $s(x)\ne 0$. The action of $\Gamma$ on 
$X^{ss}_{\kappa}$ in turn defines its action on $X//_{\kappa}G$.  

First we shall consider orbifoldization of elliptic genus
(i.e. defining the elliptic genus of the corresponding orbifold)
in the case when GIT quotient $\E//_{\kappa} G$ is smooth.

\begin{dfn}\label{phasequotients} ({\rm Orbifoldization of smooth phases}).
Let  $X,\E,G, \kappa$ be as in Definition \ref{defphases}, 
$\Gamma$ be a finite group of automorphisms of bundle $\E \rightarrow
X$ normalizing linearization $\kappa$
and $\E//_{\kappa}G$ be the phase corresponding to 
$X,\E,\kappa$ endowed with the action of $\Gamma$ 
induced from the action on $G$-semistable locus in $\E$ corresponding 
to $\kappa$. 
If $\E//_{\kappa}G$ is smooth and $F$ is 
a compact component the fixed point set of the $\CC^*$ action on $\E//_{\kappa}G$
induced by dilations $\lambda (v)=\lambda \cdot v,  v \in \E,
\lambda \in \CC^*$ 
then the $\Gamma$-orbifoldized elliptic genus of this phase 
corresponding to $F$ is the contribution (\ref{equivariantorbifoldcontribution})
of component $F$ into $\CC^*$-equivariant $\Gamma$-orbifold elliptic genus of 
$\E//_{\kappa}G$. 
\end{dfn} 

More generally, in the case when $(\E//_{\kappa}G)$
is an orbifold, assume further that it is a global quotient 
admitting as uniformization $(Y,\Gamma,T)$ in the sense of Definition \ref{uniformization}
and that there is a finite group $\Delta$ of automorphisms of  $Y$, containing  
$\Gamma$ as a normal subgroup \footnote{in particular $\Delta$ acts on
  the quotient $Y/\Gamma$},  
with action of $\Delta$ commuting with 
the action of $T$. 
We want to describe the $\Delta/\Gamma$-orbifold 
elliptic genus attached to $\E//_{\kappa}G=Y/\Gamma$
for the action induced by the action of $\Delta$.

Let $F_Y$ be the preimage in uniformization of a component of the fixed point set
$F \subset (\E //_{\kappa}G)$. Then the $\Delta$-orbifoldized 
 contribution of $F$ is the sum over all connected components $Q$ in $F_Y$
of $\Delta$-orbifoldized contributions of components $Q$ into
equivariant elliptic genus of $Y$ 
as described in Definition \ref{equivariantorbifoldcontribution}. 
More precisely, let $Q^{g,h}$ be fixed point set of pair of commuting
elements $g,h \in \Delta$ acting on $Q$, 
$V_{\lambda} \subset T_Y\vert_{Q^{g,h}}$ 
be the eigenbundle of the subgroup $<g,h>$ of $\Delta$ generated by
$(g,h)$,
$\Lambda$ is the full set of such eigenbundles in $T_{Y}\vert_{Q^{g,h}}$,
$\Lambda_{g,h}=\{\lambda \} \subset \Lambda$ is the set of
eigenbundles 
in the normal bundle to $Q^{g,h}$ in $Q$ such that $\lambda(g)=\lambda(h)=0$.
 Since we assume that the actions of $\Delta$ and $T$ commute, bundles
$V_{\lambda}$ are the eigenbundles of $T$ as well. 
Let $x_{\lambda}^T$ be $T$-equivariant Chern classes of $V_{\lambda}$
written in terms of the characters of $T^{eff}$, which is the quotient of $T$
acting effectively on the orbifold $\E//_{\kappa}G=Y/\Gamma$. 
\begin{dfn}\label{phaseorbifodization}
  $\Delta$-orbifoldized contribution of component $Q$ into equivariant elliptic genus 
of $\E_{\kappa}//G$ is given as follows:
\begin{equation}\label{orbgenorbphase} 
{1 \over {\vert \Delta \vert}}
\sum_{gh=hg,g,h \in \Delta} (\prod_{\lambda \in \Lambda_{Q^{g,h}}}
  x_{\lambda})\prod_{\lambda \in \Lambda_{Q}}
\Phi^{T^{eff}}_{Q^{g,h}}(x_{\lambda},g,h,z,\tau,\Delta)[Q^{g,h}]
\end{equation}
where 
$$
\Phi^{T^{eff}}_{Q^{g,h}}(x^T_{\lambda},g,h,z,\tau,\Delta)
=
{{\theta({x^T_{\lambda}\over {2 \pi
             i}}+\lambda(g)-\tau \lambda(h)-z)} \over 
  {\theta({x^T_{\lambda}\over {2 \pi
             i}}+\lambda(g)-\tau \lambda(h))}} e^{2 \pi i z
       \lambda(h)}
$$
\end{dfn}


The next final section contains examples showing how these definitions 
yields the invariants of Calabi Yau and Landau-Ginzburg models 
which already appeared in the literature as well as explicite examples
of some hybrid models. 

\section{Calculations of elliptic genera of phases and their
  specializations.}


\subsection{Elliptic case: weighted projective spaces and LG models}
 
The following is continuation of examples \ref{lgexample} and 
\ref{lgexample2} giving explicite form of elliptic genera of
corresponding phases and their specializations.
We shall start with the case of GIT quotient from Example 
\ref{lgexample2} i.e. Example \ref{lgexample}
with $w_i=1,D=n$.

\begin{prop}\label{lgexample3a}
 Consider $\CC^*$-action on $\CC\times \CC^n$ 
given by:
  $$  \lambda(s,z_1,....,z_n)=(\lambda^{-n}s,\lambda z_1,....,\lambda z_n) $$
There are two GIT quotients corresponding to linearizations
$\psi(\lambda)=\lambda^r$  with $r>0$ (called $\sigma$-model phase)
biholomorphic to the total space $\O_{\PP^{n-1}}(-n)$ of canonical bundle of $\PP^{n-1}$
and $r<0$ (called Landau-Ginzburg phase) biholomorphic to
$\CC^n/\mu_n$.

\par 1. The trivial sector of elliptic genus of Landau Ginzburg phase
is given by  
\begin{equation}\label{elluntwistedw=1}
          (-1)^n({{\theta(z(1-{1 \over n}))} \over {\theta ({z \over n})}})^n 
\end{equation}

\par 2.The elliptic genus of Landau-Ginzburg phase is
given by:
\begin{equation}\label{orbgenlgw=1}
{1 \over {n}}\sum_{0 \le a,b <n} (-{{\theta((1-{1 \over
      n})z+{{(a-b\tau)} \over n})}
\over {\theta({z\over n}+{{(a-b \tau)} \over n})}})^n\ee^{2 \pi \ii {{bz} }}
\end{equation}
\par 3. The elliptic genus of $\sigma$-model phase is given by: 
\begin{equation}\label{orbgenhypersurface1}
   (x{{\theta( {x \over {2 \pi i}}-z)} \over {\theta ({x \over {2 \pi
           i}})}})^{n-1}({\theta({nx \over {2 \pi i }}) \over
     {\theta({nx \over {2 \pi i }}-z)}})[\PP^{n-1}]
\end{equation}
and coincides with the elliptic genus of smooth hypersurface of degree
$n$ in $\PP^{n-1}$.
\par 4. (LG-CY correspondence) The elliptic genera
(\ref{orbgenhypersurface1})
and (\ref{orbgenlgw=1}) of $\sigma$ and LG models repsectively coincide.
\end{prop}

\begin{proof} Calculation of GIT quotients was already made in Example
  \ref{lgexample2}. The uniformization is given by $W \rightarrow
  W/\mu_n$ with $\CC^*$-action given by dilations of $W$.
The normal bundle of the fixed point, i.e. the origin $\O$ is direct 
sum of lines with equivariant Chern class being ${{u} \over n}$ where 
$u$ is the infinitesimal character of $\CC^*$ acting effectively on $W/\mu_n)$.
Hence contribution of the origin $\Ell^{\CC^*}_{\O}(W,\mu_n)$ 
into equivariant elliptic genus is given by 
$${1 \over {n}}\sum_{0 \le a,b <n} (-{{\theta({1 \over {2 \pi i}}
{u \over n} -z+{{(a-b\tau)} \over n})}
\over {\theta({1 \over {2 \pi i}}
{u\over n}+{{(a-b \tau)} \over n})}})^n\ee^{2 \pi \ii {{bz} }}
$$ 
which for $u=2\pi i z$ gives (\ref{orbgenlgw=1}). 
For $a=b=1$  one obtains (\ref{elluntwistedw=1}).

In the case of  $\sigma$-model, the $\CC^*$-action is the action 
via dilations on the fibers of the total space of
$\O_{\PP^{n-1}}(-n)$. The tangent bundle, of this total space $[\O_{\PP^{n-1}}(-n)]$,
restricted 
to the fixed point set, i.e. the zero section, get contributions from 
the tangent bundle to $\PP^{n-1}$ and from line bundle $\O_{\PP^{n-1}}(-n)$.
The equivariant Chern polynomial of the tangent bundle to
$\PP^{n-1}$ is $(1+x)^n$ and the equivatiant Chern class of $\O_{\PP^{n-1}}(-n)$ 
is ${-nx+u \over {2 \pi i}}$.
Hence the contribution of the fixed point set is:
$$(x{{\theta( {x \over {2 \pi i}}-z)} \over {\theta ({x \over {2 \pi
           i}})}})^{n-1}({\theta(-{nx \over {2 \pi i }}+{u \over {2 \pi
         i}}-z) \over 
     {\theta(-{nx \over {2 \pi i }}+{u \over {2 \pi i}})}})[\PP^{n-1}]
$$
Since $\theta(z)$ is an odd function, for $u=2 \pi i z$ we obtain
(\ref{orbgenhypersurface}).  
Since the Chern roots of a hypersurface $V_{n-2}^n$ of degree $n$ in 
$\PP^{n-1}$ are found from relation $c(V_{n-2})={{(+x)^n}\over
  {(1+nx)}}\vert_{V_{n-2}^n}$ it follows that elliptic genus of hypersurface is given
by 
$$(x{{\theta( {x \over {2 \pi i}}-z)} \over {\theta ({x \over {2 \pi
           i}})}})^{n-1}({\theta({nx \over {2 \pi i }}) \over
     {nx \cdot \theta({nx \over {2 \pi i }}-z)}})[V_{n-2}^n]$$
The latter coincides with (\ref{orbgenhypersurface}) 
since $[V_{n-2}^n]=nx \cap [\PP^{n-1}]$.

The LG/CY correspondence follows from McKay correspondence since 
contraction $[\O_{\PP^{n-1}}(-n)] \rightarrow \CC^n/\mu_n$ is a
crepant morphism.
\end{proof}

In the case when the action in Example (\ref{lgexample}) 
has arbitrary weights we obtain the following: 

\begin{prop}\label{lgexample3}
 Consider $\CC^*$-action on $\CC\times \CC^n$ 
with weights $w_1,...,w_n$ ($w_i\ \in \ZZ$ pairwise relatively prime)
and degree $D \in \ZZ_{>0}$
given by (\ref{weightedaction}):
  $$  \lambda(s,z_1,....,z_n)=(\lambda^{-D}s,\lambda^{w_1}z_1,....,\lambda^{w_n}z_n) $$
There are two GIT quotients corresponding to linearizations
$\psi(\lambda)=\lambda^r$  with $r>0$ (called $\sigma$-model phase)
and $r<0$ (called Landau-Ginzburg phase) respectively.

\par 1. The trivial sector of elliptic genus of Landau Ginzburg phase
is  
\begin{equation}\label{elluntwisted}
             \Prod_j {{{\theta({{w_jz} \over D}-z)} \over {\theta ({{w_jz}\over D})}}} 
\end{equation}

\par 2.The elliptic genus of Landau-Ginzburg phase is
given by:
\begin{equation}\label{orbgenlg}
{1 \over {D}}\sum_{0 \le a,b <D}\Prod_{i=1}^{i=n}
{{\theta(({w_i \over D}-1\
)z+
{{w_i(a-b\tau)} \over D})}
\over {\theta({zw_i \over D}+{{w_i(a-b \tau)} \over D})}}
\ee^{2 \pi \sqrt{-1} {{bw_iz} \over D}}
\end{equation}
\par 3. Let $\Gamma=\mu_{w_1}\times .... \times \mu_{w_n}$ be product
of group of roots of unity acting coordinate-wise on $\PP^{n-1}$.
Then with $x \in H^2(\PP^{n-1},\ZZ)$ being the positive generator
and with notations used in (\ref{commutingpairselliptic})
the elliptic genus of $\sigma$-model phase is given by
\begin{equation}\label{orbgenhypersurface}
 {1 \over {\vert \Gamma \vert}} \sum_{gh=hg,C} (\Prod_{\lambda(g)=\lambda(h)=0}x_{\lambda})
\Prod_{\lambda}
{{\theta({x_{\lambda}\over {2 \pi
             i}}+\lambda(g)-\tau \lambda(h)-z)} \over 
  {\theta({x_{\lambda}\over {2 \pi
             i}}+\lambda(g)-\tau \lambda(h))}} e^{2 \pi i z
       \lambda(h)} {{\theta({Dx \over {2 \pi \ii}})}
\over {\theta({Dx \over {2 \pi \ii}}
-z,\tau)}}
[{(\PP^{n-1})}^{g,h}] 
\end{equation}
(sum is taken over connected components $C$  of the fixed point sets of
pairs $g,h$).
\par 4. (LG-CY correspondence) If $\sum_1^n w_i=D$ then the elliptic genus of LG
model is equal to the orbifold elliptic genus of the hypersurface 
of degree $D$ in the weighted projective space $\PP(w_1,....,w_n)$ 
i.e. the $\Gamma$-orbifoldized elliptic genus of hypersurface of degree $D$ in
$\PP^{n-1}$ invariant under the action of the group $\Gamma$. 
\end{prop}

\begin{proof} Semistable loci corresponding to two linearizations of
  $\CC^*$ action (\ref{weightedaction}) are $s \ne 0$ and $\sum_i \vert
  z_i^2\vert \ne 0$. The quotient of the first locus is the quotient 
of $\CC^n$ by the action of $\mu_D$ and gives the Landau-Ginzburg 
phase. Parts 1 and 2 follows directly from Definition
\ref{defellgenphase} using uniformization $W$ as used in (\ref{lgexample}).

The quotient of the second locus has projection onto $\CC^n\setminus
0/\CC^*$ with action on $\CC^n\setminus 0$ being the restriction of the
action (\ref{weightedaction}). Hence this GIT quotient can be
identified with the orbifold bundle 
over weighted projective space. Using its presentation as the quotient of 
the total space of $\O_{\PP^n}(-D)$ by the action of 
$\mu_{w_1} \times.....\times \mu_{w_n}$ we obtain an uniformization of
this phase. $c_1^T$ of the normal bundle to the fixed point
set in uniformization is $-Dx+u$ where $u$ is the infinitesimal
character and the claim follows from Definition \ref{defellgenphase}. 
The rest of calculations is direct generalization of those in 
Proposition \ref{lgexample3a}.
\end{proof}

\begin{remark}  Though without Calabi Yau condition the equality of
  elliptic genus of LG model and $\sigma$-model fails, McKay
  correspondence for pairs (cf. \cite{BLannals}) still provides an
  expression 
for elliptic genus of LG model as the 
elliptic genus of a pair. 
\end{remark}

\subsection{Specialization of elliptic genus $q \rightarrow 0$}

Proposition \ref{lgexample3} has as immediate consequence 
the following relation between the spectrum of weighted homogeneous singularities 
and  $\chi_y$ genus of corresponding hypersurfaces. 

\begin{prop}\label{reductiontoy}  1.({\rm Trivial sector of LG models})  Specialization 
$q \rightarrow 0$ of elliptic genus of LG phase corresponding  
to the action (\ref{weightedaction}) is given by 
\begin{equation}\label{lgspecialization}
lim_{q \rightarrow 0} \Ell(LG)=y^{-n \over2}\Prod_{j=1}^n {{y^{{w_j} \over D}-y}\over {y^{w_j \over D}-1}}
\end{equation}
(where $y=exp(2 \pi i z)$).

2. ({\rm Relation between trivial sector of LG model and the spectrum}) 
Let $\{q_l \}, q_l \in \QQ$ be the Steenbrink spectrum of isolated 
singularity of a weighted homogeneous polynomial $s=f(z)$ 
with weights $w_i$ and degree $D$ 
\footnote{i.e. a polynomial $f(z_1,....,z_n)$ such that
$sf(z)$ is invariant for the action (\ref{weightedaction})} 
i.e. $q_l$ is the collection of $\mu$, where $\mu$ is the Milnor number
of $f$, rational numbers $q_l$ such that 
$exp(2 \pi i q_l)$ is an eigenvalue of the monodromy acting 
on the graded component $Gr^p_F(H^{n-1}(X_{\infty,f}))$ of the Hodge filtration 
of the limit mixed Hodge structure on the cohomology of the 
Milnor fiber of $f=0$ (with multiplicity of $q_l$ being equal to the dimension 
of the eigenspace). Here $p$ is such that the integer part $[q_i]$ is equal to $n-p-1$ 
(resp. $n-p$) if $exp(2\pi i q_l) \ne 1$ (resp. $exp(2 \pi i q_l)=1$).  
Let 
\begin{equation}\Xi(y)=y^{-n \over 2}\sum_{l=1}^{\mu} y^{q_l}
\end{equation}
Then 
\begin{equation}
      lim_{q \rightarrow 0}\Ell(LG)(y)=(-1)^n\Xi(y)
\end{equation}

3. ({\rm Orbifoldized-$\chi_y$ genus of LG model}) 
Specialization of 
elliptic genus of Landau-Ginzburg model is given by:
\begin{equation}\label{weightsquotient2} 
 {1 \over D}y^{-{n \over 2}}\sum_{0 \le a <D}(\Prod {{y^{w_i \over D}-y\omega_D^{-aw_i}} \over {y^{w_i
      \over D}-\omega_D^{-aw_i}}}+\sum_{1 \le b <D} y^{b \over D})   
\end{equation} 
where $\omega_D=e^{{2 \pi \sqrt{-1}} \over D}$.

4. In the case $w_i=1, D=n$ (i.e. Calabi Yau condition is satisfied) the specialization $q=0$ has the form:
 \begin{equation}\label{LGq=0w_i=1}
 {1 \over n} y^{-{n \over 2}} \sum_{k=0}^{n-1}[({{1 -y^{1-{1 \over n}}e^{{2 \pi \sqrt{-1} k}\over n}}
    \over {1-y^{-{1 \over n}}e^{{2 \pi \sqrt{-1}k}\over n}}})^n
  +\sum_{b=1}^{n-1} y^{b \over n}]
\end{equation}
\end{prop}

\begin{proof} Trivial sector of $\chi_y$-genus of LG model was already
  derived directly in
   Example \ref{lgexample}.
Now we shall obtain it as $q \rightarrow 0$ limit 
of the trivial sector of elliptic genus given in Part 1 of Prop. \ref{lgexample3}.
Indeed, Part 1 of proposition \ref{reductiontoy} follows from: 
\begin{equation}\label{keylimit}
    lim_{\tau \rightarrow i\infty} {{\theta({u(z,\tau)-z,\tau)} \over
        {\theta(u(z,\tau),\tau)}}}e^{2\pi \sqrt{-1}c z}={{sin\pi(u(z,\tau)-z)}
      \over {sin(\pi u(z,\tau))}}e^{2 \pi i \sqrt{-1}cz}
\end{equation}
$$    =y^{-{ 1\over 2}} {{1-ye^{-2\pi \sqrt{-1}u(z,0)}}
    \over {1-e^{-2 \pi \sqrt{-1}u(z,0)}}}y^c  
$$
where $u(z,\tau)$ is a linear in $z$ function and, as above, $y=e^{2 \pi
  \sqrt{-1}z} $.
(\ref{keylimit})  implies that the factor corresponding to ${{w_i}\over D}$ in (\ref{elluntwisted})
has $y^{-{ 1\over 2}}{{y^{{w_i} \over D}-y} \over {y^{w_i \over
      D}-1}}$ as the limit and 
(\ref{lgspecialization}) follows. Part 2, as was mentioned in
\ref{lgexample}, is a consequence of \cite{st}. 

Specialization of a summand in (\ref{orbgenlg}) with $b \ne 0$ 
gives $y^{b \over D}$ while each factor in summand with $b=0$ becomes
$y^{-{1 \over 2}}{{y^{w_i \over D}-y\omega_D^{-aw_i}} \over {y^{w_i
      \over D}-\omega_D^{-aw_i}}}$.
Applying (\ref{keylimit}) to (\ref{orbgenlg}) one obtains:
\begin{equation}
 {1 \over D}y^{-{n \over 2}}\sum_{0 \le a <D}(\Prod {{y^{w_i \over D}-y\omega_D^{-aw_i}} \over {y^{w_i
      \over D}-\omega_D^{-aw_i}}}+\sum_{1 \le b <D} y^{b \over D}).
\end{equation}
This implies 3 while 4 follows from it immediately. 
\end{proof}

\subsection{Specialization $q\rightarrow 0, y=1$}

Such specialization leads to numerical invarinats of phases.

\begin{corollary} 
1.Specialization $q=0,y=1$ of untwisted section of LG model 
is given by
\begin{equation}
      \Ell(LG)(q=0,y=1)=\Prod_j (1-{D \over w_j})
\end{equation}
i.e. up to sign coincides with the Milnor number of the weighted homogeneous singularity 
with weights $w_1,...,w_n$ and degree $D$.

2.Specialization $q=0,y=1$ of elliptic genus 
of LG model 
in the case 4 of Prop. \ref{reductiontoy} gives the
orbifoldized euler characteristic of LG model
\footnote{or ``orbifoldized Milnor number''}:
 \begin{equation}\label{eulerLG}
      {1\over D}[(1-D)^D+D^2-1]
\end{equation}
and coincides with the euler characteristic of smooth 
hypersurface of degree $n$ in $\PP^{n-1}$ 
(LG/CY correspondence for euler characteristic, recall that for
$w_i=1$ CY
condition is $n=D$).
\end{corollary}

\begin{proof}  
Contributions of either trivial or remaining sectors  follow from (\ref{LGq=0w_i=1}) 
and 
\begin{equation}
lim_{y \rightarrow 1}{{1 -y^{1-{1 \over D}}e^{{2 \pi \sqrt{-1} k}\over D}}
    \over {1-y^{-{1 \over D}}e^{{2 \pi \sqrt{-1}k}\over D}}}=
\begin{cases} 
  (1-D) \ \ k=0 \\
  1 \ \ k \ne 0 
\end{cases} 
\end{equation} 
In fact specialization of Prop.\ref{reductiontoy} part 4 gives
$ {1 \over D}((1-D)^n+(D-1)+D(D-1))$,
with the first and second summands corresponding to the 
first summand in the bracket with $k=0$ and $k\ge 1$ respectively
(since for $k>0$ each factor in the product is
equal to 1). The claim about matching the euler characteristic of LG model 
and smooth hypersurface can be seen directly, i.e. 
without use of McKay correspondence as in 4 in Prop.
\ref{lgexample3}, using the following 
formula (cf. \cite{methods}) 
for the euler characteristic of 
a smooth $(N-2)$-dimensional hypersurface of degree $D$:
\begin{equation}\label{eulerhypersurface}
 e(V^D_{N-2})={{(1-D)^N+ND-1} \over D}
\end{equation}
\end{proof}

\subsection{Orbifoldization of phases by the action of finite groups.}

In this section we illustrate the orbifoldization of elliptic genus 
of phases as defined in section \ref{quotientsphasessection}.

\begin{example}
Consider the $\sigma$-model phase corresponding to the action
(\ref{weightedaction}) 
with $w_i=1, i=1, ....,n$ and linearization with semistable locus
$\CC \times (\CC^n \setminus 0)$. The GIT quotient 
$\CC\times \CC^n\setminus 0//_{\kappa}\CC^*$  is the total space of the line
bundle denoted as  $[\O_{\PP^{n-1}}(-D)]$.  Let $\Gamma\subset SL_n(\CC)$ 
be a finite subgroup which we consider as acting on $\E=\CC\times \CC^n$ via 
$\gamma(s,v)=(s,\gamma \cdot v), \gamma \in \Gamma$.  
The orbifoldization of contribution of the only fixed component of 
$\CC^*$ action by dilations, which is the zero section of $\O_{\PP^{n-1}}(-D)$,
is given by the same formula as (\ref{orbgenhypersurface})
 but in which $\Gamma$  is an arbitrary subgroup of $SL_n(\CC)$ viewed 
 as acting on the total space of bundle $\O_{\PP^{n-1}}(-D)$.
As in the proof of part 3 Prop. \ref{lgexample3} 
we see that orbifoldization of the $\sigma$-model phase is the
$\Gamma$-orbifoldized 
elliptic genus of the hypersurface of degree $D$ in $\PP^{n-1}$.
\end{example}

\begin{example} Next we shall consider the $\Gamma$-quotients of LG models in the sense
of section (\ref{quotientsphasessection}). 
 First let us look at LG model corresponding to the case
 $w_1=....=w_n=D=1$  and its orbifoldization by the cyclic group $\Gamma=\mu_D$
 generated by the exponentail grading operator $J_W=(......,exp(2 \pi
 \sqrt{-1} {w_i \over D}),....)$. The GIT quotient corresponding to this LG phase
 is $\CC^n$ i.e. we have orbifoldization of smooth phase and elliptic
 genus of such orbifoldization coincides with the elliptic genus of LG
 models with $\CC^*$-action  (\ref{weightedaction}) as is specified  
in Definition  \ref{phasequotients}.
\end{example}

\begin{example}\label{orblgexample}
 Now we shall look at orbifoldization of arbitrary LG phase. Let
$\Delta\subset SL_n(\CC)$ be a finite subgroup containing exponential
grading operator $J_W=(......,exp(2 \pi
 \sqrt{-1} {w_i \over D},....)$\footnote{for discussion of the origins of this
   condition see \cite{borisovberglund}, Corollary 2.3.5}
 and such that $J_W$ belongs to its center. These conditions imply 
  that one can use as uniformization of LG phase with $\Delta$-action, 
the space $W$ such that for cyclic group 
$\Gamma=\mu_D$ generated by $J_W$ one has $W/\Gamma=\CC\times \CC^n//_{\kappa}\CC^*$.
Now Def. \ref{phaseorbifodization} yields the following
expression for orbifoldized LG phase:
\begin{equation}
{1 \over {\vert \Delta \vert}}\sum_{g,h \in \Delta, gh=hg}\Prod_{\lambda}
{{\theta(({w_i \over D}-1)z+
{\lambda(g)-\lambda(h)\tau})}\over {\theta({zw_i \over D}+{{\lambda(g)-\lambda(h) \tau}})}}
\ee^{2 \pi \ii {{\lambda(h)z}}}
\end{equation}
 \end{example}

The specialization $q \rightarrow 0$ of
orbifoldized phases goes as follows:

\begin{prop}\label{q=0} 
With notations as above, the elliptic genus of LG phase orbifoldized by a group $\Delta$ 
for $q \rightarrow 0$ specializes to 
\begin{equation} \label{orblgphasespec}
 \sum_{\{ h \} \in Conj(\Delta), X^h}
 y^{\sum_{\lambda, \lambda(h)\ne 0}(-{1 \over 2}+\lambda(h))} 
{1\over {\vert C(g) \vert}}\sum_{g \in Cent_{\Delta}(h)}
 \prod_{\lambda, \lambda(h)=1}y^{-{1 \over 2}}{{1-y^{(1-{w_j \over D})}exp(2 \pi i \lambda(g))}
   \over {1-y^{(-{w_j \over D})}exp(2 \pi i \lambda(g))}}  
\end{equation}
(here $X^h$ is the maximal subspace of $\CC^n$ fixed 
by a representative of a conjugacy class). In the case when $\Delta$
is abelian one has:
\begin{equation}\label{orblgphaseabelian}
{1\over {\vert G \vert}}\sum_{\{ h \} \in \Gamma, X^h}
 y^{\sum_{\lambda, \lambda(h)\ne 0}(-{1 \over 2}+\lambda(h))}
 \prod_{\lambda, \lambda(h)=1}{{y^{1 \over 2}-y^{({w_j \over D})-{1
         \over 2}} exp(2 \pi i \lambda(g))}
   \over {1-y^{({w_j \over D})}exp(2 \pi i \lambda(g))}}  
\end{equation}
\end{prop}
%
\begin{remark} The expression (\ref{orblgphasespec}) coincides with 
the one given in \cite{berglund}
 and expression (\ref{orblgphaseabelian})  coincides with the one  
given in Theorem 6 in \cite{ebeling}.
\end{remark}

\begin{proof} The term $\Phi^T(x,g,h,z,\tau,\Gamma)$ for $q
  \rightarrow 0$ has as limit:
\begin{equation}
    y^{{-1 \over 2}+\lambda(h)} \ \ if \ \ \lambda(h) \ne 0 \ \ \ {\rm resp.} \ \
    y^{-1\over 2}{{1-ye^{x+2 \pi i \lambda(g)}} \over {1-e^{x+2 \pi i \lambda(g)}}}
\end{equation}
For action corresponding to the weighted homogeneous polynomials with
weights $w_i$ and degree $D$ the equivariant Chern class of action of $\CC^*$ 
is $tw_i \over D$. This implies the proposition. 
\end{proof}

\subsection{Hybrid models}

Here we shall consider types of phases which are neither $\sigma$-models
or LG, called {\it hybrid} models (cf. \cite{n2}, \cite{chiodo}). 

\subsubsection{Complete intersection}

Sigma models corresponding to Calabi Yau complete intersections 
have hybrid counterparts rather than LG phases appearing in the case of 
 hypersurfaces. See \cite{chiodo} for alternative treatment of
 complete intersections via hybrid models.

\begin{dfn}\label{ci}{\it Phases of complete intersection}. 
Consider the $\CC^*$-action on $\CC^r \times \CC^n$ given by:
\begin{equation}
    \lambda(p_1,...,p_r,z_1,...,z_n)=(\lambda^{-q_1}p_1,...,\lambda^{-q_k}p_r,\lambda
    z_1,....,\lambda z_n)
\end{equation}
\end{dfn}
One of the GIT quotients, $Q_1$, is the total space of the bundle $\oplus
\O_{\PP^{n-1}}(-q_i)$ (corresponding to a linearization in one of
the cones in $Char(\CC^*) \otimes \QQ$) having 
as semistable locus $\CC^r(p_1,...,p_r) \times (\CC^n(z_1,....,z_n)
\setminus 0)$). 
For linearizations in the second cone the semistable locus 
is $(\CC^r\setminus 0) \times \CC^n$. The corresponding GIT quotient
$Q_2$ is the $\mu_D$-quotient of  the total space of the direct 
sum of $n$ copies of line bundles over weighted projective space 
$\O_{\PP(q_1,....,q_r)}(-1)^{\oplus n}/\mu_D$ where
$D=gcd(q_1,..,q_r)$ and $\mu_D$ is the group of roots of unity of
degree $D$ acting diagonally on the fibers of this direct sum. 

In the first case, the contribution into equivariant elliptic genus of the component 
of the fixed point set is given by 
\begin{equation}
   [{{x \theta({x \over {2 \pi \sqrt{-1}}}-z,\tau)}\over
 {\theta({x \over {2 \pi \sqrt{-1}}},\tau)}}]^n\cdot \Prod_{i=1}^r 
   {({x \theta({-q_ix \over {2 \pi \sqrt{-1}}}+u-z,\tau)}\over
 {\theta({-q_ix \over {2 \pi \sqrt{-1}}+u},\tau)})^n}[\PP^{n-1}] 
\end{equation}
where $u$ is the infintesimal generator of the equivariant cohomolgy
$H^*_{\CC^*}(p)$ of a point. For $u=z$ one obtains the elliptic genus
of smooth complete intersection of hypersurfaces of degree
$q_1,...,q_r$ in $\PP^{n-1}$.

Now let us calculate he elliptic genus 
in the second case (when one has a hybrid model cf. \cite{n2}). 
The GIT quotient $((\CC^r\setminus 0) \times \CC^n)/\CC^*
=\O_{\PP(q_1,....,q_r)}(-1)^{\oplus n}/\mu_D$ 
is a fiber space with the orbifold $\CC^n/\mu_D$ as a fiber and its  
base being the weighted projective space with the orbifold structure 
given by viewing $\PP^{r-1}(q_1,....,q_r)$ as a quotient of $\PP^{r-1}$ by the action 
of abelian group $\Gamma=\oplus_i \mu_{q_i}$.
The uniformization can be obtained  by taking quotient of the total 
space $[\O_{\PP^{r-1}}^{\oplus n}(-1)] $ of split vector bundle on $\PP^{r-1}$ 
by the action of $\Gamma$ such that projection on $\PP^{n-1}$ is 
compatible with $\PP^{r-1}\rightarrow \PP^{r-1}/\Gamma=\PP(q_1,...,q_r)$.
 The fixed point set of the action of $\CC^*$ on the GIT-quotient 
induced by action $t(p_1,...,p_r,z_1,...,z_n)\rightarrow (tp_1,..,tp_r,z_1,...,z_n)$
is $\PP^{r-1}(q_1,..,q_r)$ and for induced $\CC^*$-action on
$[\O_{\PP^{r-1}}^{\oplus n}(-1)] $
it is the zero section of this bundle. 
Hence for each pairs $(g,h)$ of elements of
$\Gamma$, 
contribution of $\CC^*$-fixed point to the summand of orbifold elliptic 
genus $\Ell^{\CC^*}_{\PP^{r-1}}([\O_{\PP^{r-1}}^{\oplus n}(-1)],\Gamma)$
corresponding to $(g,h)$ will have two factors.
One is coming from restriction of the tangent bundle 
\begin{equation}\label{restriction}
T_{[\O_{\PP^{r-1}}^{\oplus n}(-1)]}\vert_{{\PP^{r-1}}^{g,h}}
\end{equation} to the subspace of $\PP^{r-1}$ fixed by both $g,h$.
The latter coincides
with $T_{\PP^{r-1}}\vert_{{\PP^{r-1}}^{g,h}}$
This contribution is the summand $\Ell_{orb}(\PP^{r-1},\Gamma)^{g,h,C}$
of elliptic class 
$$\Ell_{orb}(\PP^{r-1},\Gamma)={1 \over {\vert \Gamma \vert}}
\sum_{g,h,C} \Ell_{orb}(\PP^{r-1},\Gamma)^{g,h,C}$$
 corresponding to pair $g,h$ and connected component $C$ of their fixed
 point set since $\PP^{r-1}$ is 
the fixed point set of $\CC^*$-action.
The quotient $T_{[\O_{\PP^{r-1}}^{\oplus n}(-1)]}\vert_{{\PP^{r-1}}^{g,h}}/T_{\PP^{r-1}}\vert_{{\PP^{r-1}}^{g,h}}$
is just $\O_{\PP^{r-1}}(-1)^n\vert_{{\PP^{r-1}}^{g,h}}$. The total
space of this bundle acted upon by the group 
$<g,h>$ considered as the automorphisms group of
$[\O_{\PP^{r-1}}(-1)^{\oplus n}]$. 
It also support the $\CC^*$-action by dilation. 
The corresponding equivariant contribution of this part of 
$T_{[\O_{\PP^{r-1}}^{\oplus n}(-1)]}\vert_{\PP^{r-1}}$ 
over connected component $C$ of ${\PP^{r-1}}^{g,h}$
is 
$$({{\theta({x \over
      {2 \pi i }}+{u \over D}-z+\lambda(g)-\tau\lambda(h))} 
\over {{\theta({x \over
      {2 \pi i }}+{u \over D}+\lambda(g)-\tau\lambda(h))}}}e^{2 \pi i \lambda(h)z})^n
$$
where $\lambda$ is the character of $<g,h>$ acting on this eigenbundle
over the connected component $C$ (term $u \over D$ reflects that contribution written in  terms of
character
of $\CC^*/\mu_D$ acting effectivly on the fibers).
The resulting elliptic genus of hybrid model hence can described as
\begin{equation}\label{ci} 
   {1 \over {\vert \Gamma \vert}} \sum_{g,h,C}
\Ell_{orb}(\PP^{r-1},\Gamma))^{g,h,C}\cdot 
({{\theta( {x \over {2\pi i }}+({1 \over
       D}-1)z+\lambda(g)-\lambda(h)\tau)} \over {\theta({x \over {2 \pi i}}+{z
         \over D}+\lambda(g)-\lambda(h)\tau)}} e^{{2 \pi i \lambda(h) z}})^n  [C]
\end{equation}
(sum over connected components $C \subset {\PP^{r-1}}^{g,h}$ 
of the fixed point sets of pairs $g,h$).
Note that this expression in the case $r=1$ becomes the elliptic genus of
LG-model since $x=0$, $\Gamma=\mu_D$ and for $g=e^{2 \pi i {a \over
    D}},h=e^{2 \pi i {b \over D}}$ one has $\lambda_{g,h}(g)={a \over D},
\lambda_{g,h}(g)={b \over D}$.

\subsubsection{Hypersurfaces in the products of projective spaces}

This material is discussed in \cite{n2} , Section 5.5.
Consider the action of $\CC^*\times \CC^*$ on 
$\CC \times \CC^n\times \CC^m$ given by
\begin{equation}
    (\lambda,\mu)(p,x_1,...,x_n,y_1,...,y_m)=
     (\lambda^{-n}\mu^{-m}p,\lambda x_1,...,\lambda x_n,\mu
     y_1,...,\mu y_m)
\end{equation}

There are 3 cones in $Char((\CC^*)^2)\otimes \QQ$ corresponding to 
linearizations with constant GIT with semistable loci respectively:
\begin{equation}
  \{\CC \times \CC^n\times \CC^m\}_{ss}=
\begin{cases}
          \CC \times (\CC^n\setminus 0) \times (\CC^m\setminus 0) , & 
       \text{Calabi Yau phase} \\
        \CC^* \times (\CC^n\setminus 0) \times (\CC^m) 
, & \text{hybrid phase} \\
       \CC^* \times (\CC^n) \times (\CC^m\setminus 0) 
, & \text{hybrid phase} 
        \end{cases}
\end{equation}
with the GIT quotients being respectively:
\begin{equation}
 \begin{cases}    [p_1^*\O_{\PP^{n-1}}(-n)\otimes
   p_2^*\O_{\PP^{m-1}}(-m)], \\
     [\oplus \O_{\PP^{n-1}}(-n)^m]/\mu_m, \\
      [\oplus \O_{\PP^{m-1}}(-m)^n]/\mu_n,
\end{cases}
\end{equation}
The respective elliptic genera are:
\begin{equation}
\begin{cases}
 [{{x \theta({x \over {2 \pi \sqrt{-1}}}-z,\tau)}\over
 {\theta({x \over {2 \pi \sqrt{-1}}},\tau)}}]^n  [{{y \theta({y \over {2 \pi \sqrt{-1}}}-z,\tau)}\over
 {\theta({y \over {2 \pi \sqrt{-1}}},\tau)}}]^n
 [{{\theta({{nx+my} \over {2 \pi \sqrt{-1}}},\tau)}\over
 {\theta({{nx+my} \over {2 \pi \sqrt{-1}}}-z,\tau)}}]
[\PP^{n-1} \times \PP^{m-1}] \\  
{1 \over n} \sum_{0 \le a,b<n}
({{\theta(-m{x \over {2 \pi i }}+({1 \over n}-1)z+{{a-b\tau}\over n},\tau)}
\over{\theta(-m{x \over {2 \pi i}}+{1 \over n}z+{{a-b\tau}\over n}
,\tau)}}e^{{2 \pi i bz}\over n})^n
({x{\theta({x \over {2 \pi i}}-z)} \over {\theta({x \over {2 \pi i }})}})^m[\PP\
^{m-1}] \\
{1 \over m} \sum_{0 \le a,b<m}
({{\theta(-{ny \over {2 \pi i }}+({1 \over m}-1)z+{{a-b\tau}\over m},\tau)}
\over{\theta(-{ny \over {2 \pi i}}+{1 \over m}z+{{a-b\tau}\over m}
,\tau)}}e^{{2 \pi i bz}\over m})^m
({x{\theta({y \over {2 \pi i}}-z)} \over {\theta({y \over {2 \pi i }})}})^n[\PP\
^{n-1}] 
\end{cases}
\end{equation}
The expression in the upper row represent the elliptic genus of 
Calabi Yau hypersurface of bidegree $(n,m)$ in 
$\PP^{n-1}\times \PP^{m-1}$.

\section{Appendix I: Theta functions}
Jacobi theta function $\theta(z,\tau), z \in \CC, \tau \in \HH$ is entire 
function on $\CC\times \HH$ where $\HH$ is the upper half plane 
\footnote{$\theta_1(z,\tau)$ or $\theta_{1,1}(z,\tau)$ are other 
common notations}  defined as the product:
\begin{equation}
\theta(z,\tau)=q^{1 \over 8}  (2 \sin \pi z)
\prod_{l=1}^{l=\infty}(1-q^l)
 \prod_{l=1}^{l=\infty}(1-q^l \ee^{2 \pi \ii z})(1-q^l \ee^{-2 \pi \ii
z})
\end{equation}
where $q=\ee^{2 \pi \ii  \tau}$. 

Its transformation law  is as follows:
\begin{equation}
\theta({z \over \tau},-{1 \over \tau})=
 -\ii \sqrt { \tau \over \ii }
\ee^{{\pi \ii z^2} \over {\tau}} \theta(z,\tau)  
\end{equation}
$$\theta (z+1,\tau)=-\theta(z,\tau), \ \ \ \theta(z+\tau,\tau )=-\ee^{-2
\pi \ii
z-\pi \ii \tau} \theta(z,\tau)$$

The derivative $\theta'(0,\tau)$ appears 
in expansion 
$\theta(z,\tau)=\theta'(0,\tau)z+{1\over 2}\theta''(0,\tau)z^2+....$
and satisfies:
\begin{equation}
  \theta'(0,\tau)=\eta^3(\tau), \ \ {\rm where} \ \  \eta(\tau)=q^{1 \over 24}\prod(1-q^n)  
\end{equation}
(Dedekind's) $\eta(\tau)$-function transforms as follows:
\begin{equation}
  \eta(-{1 \over \tau})=({\tau \over \ii})^{1 \over 2}\eta(\tau)
\end{equation}
It follows that
\begin{equation}\label{weightindextrade}
  {\theta({z \over \tau},-{1 \over \tau}) \over \theta'(0,-{1 \over \tau})}
={\ee^{{\pi \ii z^2}\over \tau} 
\over \tau}{\theta(z,\tau)\over \theta'(0,\tau)}  
\end{equation}
Let 
\begin{equation}\label{upsilon}
{\Upsilon}(x,\tau)=(1-e^{-x})\prod_{n=1}^{\infty}
{{(1-q^ne^x)(1-q^ne^{-x})}\over {(1-q^n)^2}} 
\end{equation}
and 
$$\Phi(x,\tau)=e^{x \over 2}\Upsilon(x,\tau)=
(e^{x \over 2}-e^{-{x \over
    2}})
\prod_{n=1}^{\infty}
{{(1-q^ne^x)(1-q^ne^{-x})}\over {(1-q^n)^2}} 
$$
(cf. \cite{hirzmod} p.170 and \cite{BLinvent} p.456). 
\footnote{in \cite{BLinvent}  Hirzebruch's $\Upsilon(z,\tau)$-function is denoted as
  $\Phi(z,\tau)$; Notation $\Phi(z,\tau)$ is the one used in appendix
to \cite{hirzmod} Cor.5.3.p.145.}
Hence
$$\Phi(x,\tau)=
2\sinh({x \over 2})\prod_{n=1}^{\infty}
{{(1-q^ne^x)(1-q^ne^{-x})}\over {(1-q^n)^2}} $$
(cf. \cite{hirzmod} p.117) i.e.
$$\Phi(x,\tau)={{\ii \theta({x \over {2 \pi \ii}},\tau)} \over 
{\eta^3(\tau)}}$$
(cf. \cite{BLinvent} p.461).

Weierstrass $\sigma$-function is defined by
\begin{equation} \label{sigmafunction}
\sigma(z,\tau)=z
\Prod_{\omega \ne 0 ,\omega \in {\ZZ+\ZZ \tau}}(1-{z\over
    {\omega}})e^{{z \over \omega}+{1 \over 2}({z \over \omega})^2}
\end{equation}
(cf. \cite{Chandra} p.52)
which can be used to describe $\Phi(z,\tau)$ where $z={x\over {2\pi \sqrt{-1}}}$ 
 (cf.\cite{hirzmod} p.145, Corollary 5.3):
\footnote{i.e. in terms of $x=2 \pi \sqrt{-1}z$ for which 
the lattice is $2 \pi \sqrt{-1}(\ZZ+\ZZ\tau)$ one has 
$\Phi(x,\tau)=\sigma(x,\tau)exp(-G_2(\tau)x^2)$.
Ref. \cite{hirzmod}, \cite{hirzarticle} use this notation while 
we selected traditional notations (in particular consistent with \cite{weil}).} 
\begin{equation}\label{sigmaformula}
       \Phi(z,\tau)=exp(4\pi^2G_2(\tau)
       z^2)\sigma(z,\tau)=exp(-{{e_2(\tau)}\over 2}z^2)\sigma(z,\tau)
\end{equation}
Here the quasi-modular forms $G_2(\tau)$ and $e_2(\tau)$ are given by
\begin{equation}\label{eisenstein}
G_2(\tau)=-{1 \over {24}}+\sum_{n=1}^{\infty}(\sum_{d \vert
  n}d)q^n=-{1 \over {8\pi^2}}e_2(\tau) 
\ \ \ {\rm where} \ \ e_2(\tau)=\sum_{n,(m,n)\ne (0,0)}{\sum_m}{1 \over {(m+n\tau)^2}}
\end{equation}
We also consider the following product expansion (cf. \cite{weil} Ch.4 sect.3):
\begin{equation}\label{weilphi}
      \phi(z,\tau)=x \Prod_e'(1-{z \over w})
\end{equation}
(related to (\ref{sigmafunction}); product is taken over the elements
$w$ of the lattice $W=\{1,\tau\}$,
subscript $e$ designates Eisenstein ordering of factors and $\prime$
indicates omitting
$(0,0)\in W$. $\phi(z,\tau)$  admits the following product
  formula in $q$ (cf. (15) ibid){ 
\begin{equation}
    \phi(z,\tau)={ 1\over {2 \pi \sqrt{-1}}}
{{(e^{\pi
        \sqrt{-1}z}-e^{-\pi \sqrt{-1}z}) \Prod_{n \ge 1}(1-q^n e^{2\pi
        \sqrt{-1}z})(1-q^n e^{-2\pi \sqrt{-1}z})}
\over {\Prod_{n \ge 1} (1-q^n)^2}}
\end{equation}
$$={ 1\over {2 \pi \sqrt{-1}}}\Phi(x,\tau)={1 \over {2
    \pi}}{{\theta(z,\tau)} \over {\eta^3(\tau)}}$$

\section{Appendix: Quasi-Jacobi forms}\label{appendix2}

Recall the following:

\begin{dfn}\label{jacobi}(cf. \cite{eichler}, \cite{mequasijacobi}) 
{\it Meromorphic Jacobi form} of index $t \in {1 \over 2}\ZZ$ and weight $k$ for a
finite index subgroup of the Jacobi group
$\Gamma_1^J=SL_2(\ZZ) \propto \ZZ^2$
is defined as a meromorphic in elliptic variable $z$ 
function $\chi$ on $\HH \times \CC$ 
having expansion
$\sum c_{n,r}q^n\zeta^r$ in $q=exp (2 \pi \sqrt{-1} \tau)$ 
and satisfying the following functional
equations:
\begin{equation}                                                                                                                               
\chi({{a\tau+b} \over {c\tau+d}},{z \over {c\tau+d}})=                                                                            
(c\tau+d)^ke^{{2 \pi i t c z^2} \over {c\tau+d}}\chi(\tau,z)
\end{equation}
\begin{equation}\label{elllaw}
\chi(\tau,z+\lambda\tau+\mu)=(-1)^{2t(\lambda+\mu)}                                                                             
e^{-2\pi i t(\lambda^2\tau+2 \lambda z)}\chi(\tau,z)                                                                              
\end{equation}
for all elements
$[\left( \begin{array}{ccc} a & b \\ c & d \end{array}                                                                            
\right ), 0 ]$ and $[\left( \begin{array}{ccc} 1 & 0 \\ 0 & 1 \end{array}                                                         
\right ), (a,b) ]$
in $\Gamma$.

A meromorphic Jacobi form is called a {\it weak 
Jacobi form} if 

a) it is holomorphic in $\HH \times \CC$ and 

b)  it has Fourier expansion 
$\sum c_{n,r}q^n\zeta^r$ in $q=exp (2 \pi \sqrt{-1} \tau)$ 
in which $n \ge 0$ 

The functional equation (\ref{elllaw}) implies that Fourier
coefficients
$c_{n,r}$ depend on $r \ {\rm mod} \ 2m$ and $\Delta=4nm-r^2$ 
({\it the discriminant}). A weak Jacobi form is called {\it Jacobi form}
(resp. {\it cusp form}) if the coefficients $c_{n,r}$ with $\Delta<0$ (resp. $\Delta \le 0$) 
are vanishing. 
\footnote{mentioning that this condition on Fourier expansion applicable in 
holomorphic case only and the restriction $n\ge 0$ were inadvertently omitted in 
\cite{mequasijacobi}.}
 
\begin{remark} Presentation (\ref{formula1}) 
  provides 
  Fourier expansion of elliptic genus having 
  non-negative powers of $q$ (i.e. yields a weak Jacobi form) while
  powers of $y$ can be negative.
\end{remark}

The algebra of Jacobi forms is the bi-graded algebra $J=\oplus J_{t,k}$.
and the algebra of Jacobi forms of index zero is the sub-algebra
$J_0=\oplus_k J_{0,k} \subset J$.
\end{dfn}

We
shall need below the following real analytic functions:
\begin{equation}\label{lambdamu}
\lambda(z,\tau)={{z-\bar z} \over {\tau-\bar \tau}},
\ \ \ \ \mu(\tau)=
{1 \over {\tau -\bar \tau}}
\end{equation}

Their transformation properties are as follows:

\begin{equation}\label{lambdamu}
 \lambda({z \over {c\tau+d}},{{a \tau+b} \over {c \tau+d}})=
(c\tau+d)\lambda(z,\tau)-2icz
\end{equation}
$$\lambda(z+m\tau+n,\tau)=\lambda(z,\tau)+m$$
\begin{equation}
\mu({{a \tau+b} \over {c \tau+d}})=(c\tau+d)^2\mu(\tau)
-2ic(c\tau+d)
\end{equation}

\begin{dfn}\label{defquasijacobi} {\it Almost meromorphic Jacobi form} 
of weight $k$, index zero and  a 
depth $(s, t)$ is a (real) meromorphic function in 
$\CC\{q^{1 \over l},z\}[z^{-1}, \lambda, \mu ]$, 
with $\lambda,\mu$ given by (\ref{lambdamu}), 
i.e. polynomial in $\lambda, \mu$ with complex meromorphic 
functions as coefficients 
which

a) satisfies the functional equations in Definition \ref{jacobi} of
 Jacobi 
forms of weight k and index zero and

b) which has degree at most $s$ in $\lambda$ and at most t in $\mu$.

{\it Quasi-Jacobi form of weight $k$, index zero and depth $(s,t)$} is
the term of bi-degree $(0,0)$
in $\lambda, \mu$ of an almost meromorphic Jacobi form of weigth $k$ 
and depth $(s,t)$.
{\it Algebra of quasi-Jacobi forms} is bi-graded filtered algebra generated
by filtered algebra of quasi-Jacobi forms and algebra of Jacobi forms (which
have depth $(0,0)$ and have trivial filtration).

\end{dfn}

\begin{example}\label{eisenstein} 1.{\it Two variable Eisenstein series}
  (cf. \cite{weil},\cite{mequasijacobi}).
 Consider the following, meromorphic in $z$ functions 
\begin{equation}
      E_n(z,\tau)=\sum_{a,b \in \ZZ^2}({1 \over {z+a\tau+b}})^n   \ \ \ n
      \in \ZZ, n \ge 1 
\end{equation} 
These series are absolutely convergent for $n \ge 3$ and yields
meromorphic Jacobi forms of weight $n$ and index $0$. 
For $n=1,2$
one obtains meromorphic function using Eisenstein summation
 (cf. \cite{weil}) which are quasi-Jacobi forms of index $0$, weight
 $n=1,2$ and depth $(1,0)$ for $n=1$ and $(0,1)$ for $n=2$
 (cf. \cite{mequasijacobi}). 
$E_2-e_2$ is Jacobi form (here $e_2(\tau)$ is quasi-modular form which is 
the one variable Eisenstein series).

The products 
\begin{equation}\label{eisensteinholo}
\hat E_n(z,\tau)=E_n(z,\tau)({\theta(z,\tau) \over {\theta'(0,\tau)}})^n  \ \ \ (n \ne 2)
\ \ \ 
\hat E_2=(E_2(z,\tau)-e_2(\tau))({\theta(z,\tau) \over {\theta'(0,\tau)}})^2
\end{equation}
are holomorphic quasi-Jacobi forms (Jacobi forms for $n \ge 2$.
\end{example}

The structure of the algebra of quasi-Jacobi forms generated by forms
(\ref{eisensteinholo}) is as follows. 
\begin{theorem} The algebra $QJac_{0,*}$ (or simply $QJac$) 
of quasi-Jacobi forms of weight zero and 
  index ${d\over 2}, d \in \ZZ^{\ge 1}$ is polynomial 
algebra with generators $\hat E_n, n=1,2,3,4$.
The algebra $Jac_{0,{*}}$ of Jacobi forms of weight zero and index $d \over 2$ 
(or $Jac$) is polynomial algebra in three generators $\hat E_2,\hat E_3,\hat E_4$. 

The algebra $QJac$ is isomorphic to the algebra of complex cobordisms
 $\Omega^U$ modulo the ideal $I$ generated by $X_1-X_2$ where $X_1,X_2$ are
$K$-equivalent.
The algebra $Jac$ is isomorphic to the algebra $\Omega^{SU}$ of complex cobordisms of
manifolds 
with trivial first Chern class modulo the ideal $I \cap \Omega^{SU}$.
\end{theorem}

\begin{remark} 1. Different generators of the algebra $Jac$ are described
  in \cite{gritsenko}. 

2.Term ``quasi-Jacobi forms'' used in \cite{ober} in
a slightly more narrow sense 
than in \cite{mequasijacobi} and above, where 
author apparently was unaware of \cite{mequasijacobi}.
Quasi-Jacobi forms considered in \cite{ober}  belong to the 
algebra generated by the function:
\begin{equation}
      {\theta(z,\tau) \over {\eta^3(\tau)}}, {{\partial log
          ({\theta(z,\tau) \over {\eta^3(\tau)}})} \over {\partial
          z}}, e_2(\tau),e_4(\tau),\wp(z,\tau), y{{d\wp(z,\tau)}\over {dy}}
\end{equation}
($y=exp(2 \pi \sqrt{-1}z)$) are in the algebra of meromorphic quasi-Jacobi forms as defined 
in (\ref{defquasijacobi}) (cf. also \cite{mequasijacobi}). 
Indeed  ${{\partial ({\theta(z,\tau) \over {\eta^3(\tau)}})} \over
  {\partial z}}=E_1(z,\tau)$
(it follows from Appendix I, also cf. \cite{weil} ch.IV,sect.3 (15))
and 
also $\wp(z,\tau)=E_2-e_2, y{{d\wp(z,\tau)} \over
 {dy}} =-2E_3(z,\tau)$ and 
modular functions are clearly part  of the algebra described in Def. \ref{defquasijacobi}. 
\end{remark}

\section{Acknowledments}

The material of this paper was reported on several conferences
including Durham, Lausanne and Toronto. I want to thank organizers of
these meeting and  all those who commented on my reports
The author was supported 
by a grant from Simons Foundation.

\end{document}